\documentclass[11pt]{article}
\usepackage[english]{babel}
\usepackage{amsmath,amsthm}
\usepackage{amsfonts}
\usepackage{mathrsfs}
\usepackage{color}
\usepackage{bbm}
\usepackage{enumerate}
\usepackage{amsrefs}

\newtheorem{theorem}{Theorem}[section]

\newtheorem{lemma}[theorem]{Lemma}

\newtheorem{assumption}[theorem]{Assumption}
\theoremstyle{definition}
\newtheorem{definition}[theorem]{Definition}
\theoremstyle{remark}
\newtheorem{remark}[theorem]{Remark}

\numberwithin{equation}{section}
\begin{document}
\def\Pro{{\mathbb{P}}}
\def\E{{\mathbb{E}}}
\def\e{{\varepsilon}}
\def\veps{{\varepsilon}}
\def\ds{{\displaystyle}}
\def\nat{{\mathbb{N}}}
\def\Dom{{\textnormal{Dom}}}
\def\dist{{\textnormal{dist}}}
\def\R{{\mathbb{R}}}
\def\O{{\mathcal{O}}}
\def\T{{\mathcal{T}}}
\def\Tr{{\textnormal{Tr}}}
\def\I{{\mathcal{I}}}
\def\A{{\mathcal{A}}}
\def\H{{\mathcal{H}}}
\def\S{{\mathcal{S}}}
\newcommand{\Om}{\Omega}
\newcommand{\clf}{{\mathcal{F}}}
\newcommand{\cle}{{\mathcal{E}}}
\newcommand{\clg}{{\mathcal{G}}}
\newcommand{\cla}{{\mathcal{A}}}
\newcommand{\cll}{{\mathcal{L}}}
\newcommand{\clp}{{\mathcal{P}}}
\newcommand{\cli}{{\mathcal{I}}}
\newcommand{\PP}{{\mathbb{P}}}
\newcommand{\RR}{{\mathbb{R}}}
\newcommand{\NN}{{\mathbb{N}}}
\newcommand{\bfu}{\mathbf{u}}
\newcommand{\eps}{\epsilon}

\title{Equivalences and counterexamples between several definitions of the uniform large deviations principle}%
\author{M. Salins \footnote{
Boston University, Department of Mathematics and Statistics}}
\maketitle
\begin{abstract}
  This paper explores the equivalences between four definitions of uniform large deviations principles and uniform Laplace principles found in the literature. Counterexamples are presented to illustrate the differences between these definitions and specific conditions are described under which these definitions are equivalent to each other. A fifth definition called the equicontinuous uniform Laplace principle (EULP) is proposed and proven to be equivalent to Freidlin and Wentzell's definition of a uniform large deviations principle. Sufficient conditions that imply a measurable function of infinite dimensional Wiener process satisfies an EULP using the variational methods of Budhiraja, Dupuis and Maroulas are presented. This theory is applied to prove that a family of Hilbert space valued stochastic equations exposed to multiplicative noise satisfy a uniform large deviations principle that is uniform over all initial conditions in bounded subsets of the Hilbert space. This is an improvement over previous weak convergence methods which can only prove uniformity over compact sets.
\end{abstract}

\section{Introduction}
The theory of large deviations principles, developed in the 1960s by Freidlin, Wentzell, Varadhan and others, characterizes the asymptotic decay rate of rare probabilities. There are a several manuscripts on the theory of large deviations including \cite{dz,de,fk,fw,v}. One setting for the problem is as follows. Let $(\mathcal{E},\rho)$ be a Polish space.   Let $\{X^\e\}_{\e>0}$ be a collection of $\mathcal{E}$-valued random variables, let $a(\e)$ be a positive real-valued function with the property that $\lim_{\e \to 0} a(\e) = 0$ and let $I: \mathcal{E} \to [0,+\infty]$ be a lower semi-continuous function. A family of $\mathcal{E}$-valued random variables $\{X^\e\}_{\e>0}$ is said to satisfy a large deviations principle with respect to a rate function $I$ and speed $a(\e)$ if \cite{dz}
\begin{enumerate}[(a)]
  \item For any open $G \subset \mathcal{E}$,
  \begin{equation} \label{eq:LDP-open}
    \liminf_{\e \to 0} a(\e) \log \Pro(X^\e \in G) \geq -\inf_{\varphi \in G} I(\varphi)
  \end{equation}
  \item and for any closed $F \subset \mathcal{E}$,
  \begin{equation} \label{eq:LDP-closed}
    \limsup_{\e \to 0} a(\e) \log \Pro(X^\e \in F) \leq - \inf_{\varphi \in F} I(\varphi).
  \end{equation}
\end{enumerate}
By Theorem 4.2 of \cite{comman-2003} (see also Theorems 1.2.1 and 1.2.3 of \cite{dz}), the large deviations principle is equivalent to the so-called Laplace principle, which says that for any bounded and continuous $h: \mathcal{E} \to \mathbb{R}$,
\begin{equation} \label{eq:LP}
  \lim_{\e \to 0} a(\e) \log \E \exp \left(\frac{h(X^\e)}{a(\e)} \right) = - \inf_{\varphi \in \mathcal{E}} \{h(\varphi) + I(\varphi)\}.
\end{equation}
For any $s\geq 0$ define the level sets of $I$ by $\Phi(s) := \{\varphi \in \mathcal{E}: I(\varphi) \leq s\}$. If $\Phi(s)$ is a compact subset of $\mathcal{E}$ for any $s\geq 0$, then $I$ is called a good rate function. If $I$ is a good rate function, then an equivalent formulation of the large deviations principles \cite[Theorem 3.3.3]{fw} is the following formulation by Freidlin and Wentzell
\begin{enumerate}[(a)]
 \item For any $\delta>0$ and $s_0>0$,
 \begin{equation} \label{eq:FWLDP-lower}
   \liminf_{\e \to 0} \inf_{\varphi \in  \Phi(s_0)} \left( a(\e) \log \Pro(\rho(X^\e,\varphi)<\delta) + I(\varphi) \right) \geq 0.
 \end{equation}
 \item For any $\delta>0$ and $s_0>0$,
 \begin{equation} \label{eq:FWLDP-upper}
   \limsup_{\e \to 0} \sup_{s \in [0,s_0]} \left( a(\e) \log \Pro(\dist(X^\e,\Phi(s))\geq \delta) + s \right) \leq 0.
 \end{equation}
 where for any point $\varphi \in \mathcal{E}$ and any set $B \subset \mathcal{E}$, we defined the distance function by $\dist(\varphi, B) =\inf_{\psi \in B} \rho(\varphi,\psi)$.
\end{enumerate}

While these three formulations of the large deviations principle are all known to be equivalent, the situation is more complicated when the random variables depend on another parameter in addition to $\e$. As a motivating example, consider the family of small noise stochastic differential equations
\begin{equation} \label{eq:SDE}
  dX^\e_x(t) = b(X^\e_x(t)) dt + \sqrt{\e} \sigma(X^\e_x(t))dW(t), \ \ \ \ \ \  X^\e_x(0) = x \in \mathbb{R}^d.
\end{equation}
In the above equation, $W(t)$ is a $d$-dimensional Wiener process, $b:\mathbb{R}^d \to \mathbb{R}^d$ is a Lipschitz continuous vector field and $\sigma: \mathbb{R}^d \to \mathbb{R}^{d\times d}$ is a Lipschitz continuous $d\times d$ matrix valued function.
Notice that the $X^\e_x$ are indexed both by the size of the noise $\e$ and the initial condition $x$. We consider $X^\e_x$ as $\mathcal{E}=C([0,T]:\mathbb{R}^d)$-valued random variables where $C([0,T]:\mathbb{R}^d)$ is the space of continuous $\mathbb{R}^d$-valued functions endowed with the supremum norm.



For several applications, such as characterizing the exit time of $X^\e_x$ from a domain, the large deviations of $X^\e_x$ must be \textit{uniform} with respect to the initial conditions in certain subsets of the space \cite{dz,fw}. In this paper we compare several definitions of uniform large deviations principles that are found in the literature.  The first definition of a uniform large deviations principle is due to Freidlin and Wentzell \cite{fw} (Definition \ref{def:fwuldp} below). We will call this the Freidlin-Wentzell uniform large deviations principle (FWULDP). The next definition can be found in \cite{dz} (Definition \ref{def:dzuldp} below). We will call it the Dembo-Zeitouni uniform large deviations principle (DZULDP). The third definition is called the uniform Laplace principle (ULP) and can be found in \cite{de} (Definition \ref{def:ulp} below).

Each of these definitions has been widely used in the literature. The following lists are references are by no means complete, but they give examples of the wide varieties of problems in which these different definitions of uniform large deviations have been used. The FWULDP has been used in the work of Cerrai and R\"ockner \cite{cr-2004}, Peszat \cite{p-1994}, and Sowers \cite{s-1992}. The DZULDP has been used by Chenal and Millet \cite{cm-1997}, Gautier \cite{g-2005}, and Veretennikov \cite{v-2000}. A very general weak convergence approach that is sufficient to prove the uniform Laplace principle was introduced by Budhiraja, Dupuis, and Maroulas \cite{bdm-2008}. Since then, the ULP has been used by many authors including Budhiraja and Biswas \cite{bb-2011}, Wu \cite{w-2011}, and Cai, Huang  and Maroulas \cite{chm-2015}.

The main question of this paper is whether the FWULDP, the DZULDP, and the ULP are equivalent. Without further assumptions, the answer is no. In section \ref{S:counterexamples} we illustrate this lack of equivalence with  simple counter-examples. We study stochastic processes $X^\e_x(t) = x + \sqrt{\e}W(t)$, where $x \in \mathbb{R}$ and $W(t)$ is a one-dimensional Brownian motion. First, we show in Theorem \ref{thm:BM-FWULDP} that $\{X^\e_x\}$ satisfies a FWULDP that is uniform over $x$ in the whole space. On the other hand, $X^\e_x$ does not satisfy either a DZULDP or a ULP over the whole space (Theorems \ref{thm:counter-BM-DZ} and \ref{thm:ulp-counter}). In fact, the DZULDP fails to hold for $X^\e_x$ uniformly over $x$ in a set $A$ if $A$ fails to be compact. We give an example where the DZULDP fails to hold for the bounded, pre-compact, but not compact set $A = \{2^{-n}\}_{n \in \NN}$ (Remark \ref{rem:counterexample-bounded}). These counterexamples prove that these three definitions are not exactly the same. Their equivalences requires certain compactness criteria.

The general setting for this problem is to let $(\mathcal{E},\rho)$ be a Polish space and $\mathcal{E}_0$ be a set used for indexing. At first, we make no assumptions about topology on $\mathcal{E}_0$. We consider a family of $\mathcal{E}$-valued random variables $\{X^\e_x\}$ indexed by $\e>0$ and $x \in \mathcal{E}_0$. For each $x \in \mathcal{E}_0$ there is a function $I_x: \mathcal{E} \to [0,+\infty]$ called a rate function. For $x \in \mathcal{E}_0$ and $s\geq 0$, the level sets of $I_x$ are denoted by $\Phi_x(s) = \{\varphi \in \mathcal{E}: I_x(\varphi) \leq s\}$.

In Theorem \ref{thm:FWULDP=ULP}, we prove that the FWULDP and the ULP are equivalent under the assumption that $\bigcup_{x \in A} \Phi_x(s)$ is a pre-compact subset of $\mathcal{E}$ for any $A \in \mathscr{A}$ and $s\geq0$ (Assumption \ref{assum:compact-level-union}). Neither the definitions of the FWULDP and ULP, nor their equivalence theorem require any kind of topology on the index set $\mathcal{E}_0$. The equivalence between FWULDP and the DZULDP, on the other hand, requires that $\mathcal{E}_0$ be metrizable and that whenever $x_n \to x$ in $\mathcal{E}_0$, the level sets $\Phi_{x_n}(s)$ converge to $\Phi_x(s)$ in an appropriate Hausdorff metric (Assumption \ref{assum:compact-initial-conds}). Under that assumption along with the assumption that $\mathscr{A}$ is the collection of compact subsets of $\mathcal{E}_0$, the FWULDP and DZULDP are equivalent (Theorem \ref{thm:FWULDP=DZULDP}). In the case where $x$ encodes the initial condition of a stochastic process $X^\e_x$, Assumption \ref{assum:compact-initial-conds} requires that $\mathscr{A}$ contains only compact sets of initial conditions and Assumption \ref{assum:compact-level-union} requires that $\mathscr{A}$ contains only pre-compact sets of initial conditions.

In the setting of finite dimensional stochastic differential equations such as \eqref{eq:SDE}, the restriction to compact or pre-compact subsets of initial conditions is usually not terribly restrictive. For example, when studying the exit time of $X^\e_x$ from a bounded domain, it is sufficient to prove uniformity of the large deviations principle over initial conditions in compact sets because all closed bounded sets are compact. In infinite dimensional spaces, on the other hand, bounded sets are not generally compact. Furthermore, compact subsets of infinite dimensional Banach spaces have no interior. This means that compact sets are not very helpful for studying exit problems because exterior points of a compact set are arbitrarily close to every element of the set. The reliance on the compactness or pre-compactness of sets of initial conditions when using the ULP and DZULDP demonstrates some limitations of these two approaches.

There are various possible modifications to the ULP and DZULDP that remove this reliance on compactness. Recently, David Lipshutz \cite{l-2017} studied exit problems for stochastic delay equations with small noise. The initial conditions belong to the space of continuous functions $C([-\tau,0])$, which is an infinite dimensional space.  To prove the exit time asymptotics, Lipshutz proposed a modification of the DZULDP that we call the LULDP (Definition \ref{def:luldp} below).
This definition fixes the problems pointed out by our counterexamples and in particular, the LULDP can be valid for $A$ that are not compact. Unfortunately, the LULDP is not equivalent to the FWULDP as we show in Theorem \ref{thm:luldp-counter}. This counterexample involves the process $Y^\e_x(t) = (1+\e)x + \sqrt{\e}W(t)$, which does not satisfy a FWULDP but does satisfy a LULDP over the whole space.

The compactness of $\bigcup_{x \in A} \Phi_x(s)$ is required to prove the equivalence between the FWULDP and the ULP precisely because continuous functions on compact sets are uniformly continuous. When this compactness is lacking, as is the case for $X^\e_x(t) = x + \sqrt{\e}W(t)$ with $A = \mathbb{R}$, we build our counterexample by choosing a function $h: \mathcal{E} \to \mathbb{R}$ that is continuous, but not uniformly continuous. Based on this observation, we propose the new definition of the equicontinuous uniform Laplace principle (EULP) (Definition \ref{def:eulp} below). The equicontinuous uniform Laplace principle is like the uniform Laplace principle with the added requirement that the limit must also be uniform over any family of equibounded, equicontinuous test functions from $\mathcal{E} \to \mathbb{R}$.

We show in Theorem \ref{thm:FWULDP=EULP} that the EULP and FWULDP are equivalent with no extra assumptions. In particular, this equivalence does not require the compactness of initial conditions or of level sets.
The benefit of the EULP is that it can be proven via the variational methods of Budhiraja, Dupuis, and collaborators \cite{bd-1998,bd-2000,bdm-2008}. In those papers, they used a variational method to study the uniform Laplace principle for a family of measurable mappings of infinite dimensional Wiener processes. The method was sufficient for proving that a ULP held uniformly with respect to initial conditions in compact sets. In Section \ref{S:EULP} we modify this method to be applicable for initial conditions that are not in compact sets. Specifically, in \cite{bdm-2008}, Budhiraja, Dupuis, and Maroulas assumed that for all $\e\geq 0$ there were measurable mappings $\mathscr{G}^\e$, such that $X^\e_x = \mathscr{G}^\e(x, \sqrt{\e}\beta)$, where $\beta$ is some infinite dimensional Wiener process. They assume that if $x_n \to x$ and $u_n$ converge in distribution to $u$  in the weak  topology on $L^2([0,T]:H_0)$ for an appropriately defined space $H_0$, that
$\mathscr{G}^\e\left(x_n, \sqrt{\e} \beta + \int_0^\cdot u_n(s)ds \right)$ converges in distribution to  $\mathscr{G}^0\left( x, \int_0^\cdot u(s)ds\right)$.

If the initial conditions do not belong to a compact set, such a weak convergence approach is impossible. For an example, consider $X^\e_x(t) = x + \sqrt{\e}W(t)$ where we take an unbounded sequence of initial conditions. If $x_n = n$, and $\mathscr{G}^\e(x,\sqrt{\e}W) = x + \sqrt{\e}W$, then for a sequence $u_n \in L^2([0,T])$ almost surely, it is impossible for $\mathscr{G}^\e\left( x_n, \sqrt{\e}W + \int_0^\cdot u_n(s)ds \right) = n + \sqrt{\e}W + \int_0^\cdot u_n(s)ds$ to converge in distribution to anything because the initial conditions diverge.

In this paper, we do not assume that $\mathcal{E}_0$ has any topology and we do not even require that the mapping $x \mapsto \mathscr{G}^\e(x, w)$ be measurable. To emphasize this we consider for any $\e>0$ and $x \in \mathcal{E}_0$ measurable mappings $\mathscr{G}^\e_x: C([0,T]:\mathbb{R}^\infty) \to \mathcal{E}$. Instead of working with weak convergence, we require that $\mathscr{G}^\e_x\left( \sqrt{\e}\beta + \int_0^\cdot u(s)ds \right)$ converges to $\mathscr{G}^0_x\left( \int_0^\cdot  u(s)ds \right)$ in probability uniformly with respect to $x$ and $u$. Specifically, we prove that the EULP will hold if for any $\delta>0$,
\[\lim_{\e \to 0}\sup_{x \in A} \sup_{ u \in \mathcal{P}_2^N}  \Pro \left(\rho \left( \mathscr{G}^\e_x \left( \sqrt{\e} \beta + \int_0^\cdot u(s)ds \right), \mathscr{G}^0_x \left( \int_0^\cdot u(s)ds \right) \right)>\delta \right) = 0,\]
where $\mathcal{P}_2^N$ is a family of progressively measurable processes in an appropriate space whose $L^2$ norms are bounded by $N$ with probability one (Assumption \ref{assum:mathcal-G}).

In Section \ref{S:example}, we apply this theory to study the uniform large deviations of a Hilbert space valued family of stochastic process. Let $H$ be a separable infinite dimensional Hilbert space and study the mild solutions to the abstract stochastic differential equations (see \cite[Chapter 7.1.1]{dpz})
\[dX^\e_x(t) = [\mathcal AX^\e_x(t) + B(X^\e_x(t))]dt  + \sqrt{\e}G(X^\e_x(t))dw(t), \ \ \ X^\e_x(0) = x \in H.\]
In this equation, $\mathcal A$ is an unbounded linear operator that generates a $C_0$ semigroup on $H$ and $w(t)$ is a cylindrical Wiener process on another separable Hilbert space $U$. We show that if $B$ and $G$ are globally Lipschitz continuous in an appropriate sense, then the mild solutions to $X^\e_x$ satisfy a EULP (and therefore also a FWULDP) in $\mathcal{E}= C([0,T]:H)$ that is uniform over initial conditions in bounded subsets of $H$. Note that bounded subsets of $H$ are generally not compact. Furthermore, we show that if the multiplicative noise coefficient $G$ is bounded in an appropriate sense, then the FWULDP is uniform over initial conditions in any subset of $H$ including unbounded subsets (and including the entire space). This result demonstrates the power of the EULP because previous variational methods could only be used to prove uniformity over compact sets of initial conditions.

The outline of this paper is as follows. In Section \ref{S:results} we state the assumptions and main results of this paper. In Section \ref{S:counterexamples}, we present counterexamples to demonstrate the lack of equivalence between the FWULDP, DZULDP, ULP, and LULDP.  In Section \ref{S:example}, we use the EULP to prove that a Hilbert space valued stochastic process satisfies a FWULDP that is uniform over initial conditions in bounded (but not necessarily compact) subsets of the infinite dimensional Hilbert space. We also give a conditions under which the Hilbert space valued process satisfies a FWULDP that is uniform over initial conditions in any (including unbounded) subsets of $H$.  In Sections \ref{S:Proof-FWULDP=ULP}--\ref{S:Proof-FWULDP=EULP}, we prove the equivalence between the FWULDP and the ULP, DZULDP, and EULP under appropriate assumptions. In Section \ref{S:EULP}, we prove that uniform convergence in probability for certain measurable functionals of infinite dimensional Wiener processes implies that the processes satisfy an EULP.  In Appendix \ref{S:properties} we recall some useful properties about rate functions. Appendices \ref{S:lemmas} and \ref{S:good} include some proofs about the Hilbert space valued process from Section \ref{S:example}.

\section{Assumptions and main results} \label{S:results}
Let $(\mathcal{E},\rho)$ be a Polish space and let $\mathcal{E}_0$ be a set. For now we do not make any topological assumptions about $\mathcal{E}_0$. For any $\varphi \in \mathcal{E}$ and $B \subset \mathcal{E}$, let
\begin{equation}
  \dist(\varphi,B) = \inf_{\psi \in B} \rho(\varphi,\psi).
\end{equation}
 We recall the definition of the Hausdorff metric on nonempty closed subsets of $\mathcal{E}$. For any nonempty, closed subsets $B_1,B_2 \subset \mathcal{E}$, the Hausdorff metric is given by
\begin{equation} \label{eq:Hausdorff}
  \lambda(B_1,B_2) =  \max \left\{\sup_{\varphi \in B_1} \dist(\varphi, B_2), \sup_{\varphi \in B_2} \dist(\varphi, B_1) \right\}.
\end{equation}

The space of bounded continuous functions $h:\mathcal{E} \to \mathbb{R}$ is denoted by $C(\mathcal{E})$. This is a Banach space under the sup-norm  $\|h\|_{C(\mathcal{E})} = \sup_{\varphi \in \mathcal{E}} |h(\varphi)|.$

Let $(\Omega, \mathcal{F}, \Pro)$ be a probability space and let $\{X^\e_x: \e>0, x \in \mathcal{E}_0\}$ be a collection of $\mathcal{E}$-valued random variables. We denote the expectation in $(\Omega, \mathcal{F}, \Pro)$ by $\E$. Let $\{I_x: x \in \mathcal{E}_0\}$ be a collection of lower-semicontinuous rate functions $I_x: \mathcal{E} \to [0,+\infty]$. Let $\Phi_x(s) = \{\varphi \in \mathcal{E}: I_x(\varphi)\leq s\}$ be the level sets of $I_x$. If $\Phi_x(s)$ is a compact subset of $\mathcal{E}$ for all $s\geq0$, then $I_x$ is called a good rate function.

The first definition of a uniform large deviations principle is due to Freidlin and Wentzell and is defined at the end of Section 3.3 of \cite{fw}.
\begin{definition}[ Freidlin-Wentzell uniform large deviations principle \\(FWULDP)] \label{def:fwuldp}
  Let $\mathscr{A}$ be a collection of subsets of $\mathcal{E}_0$ and $a(\e)$ be a function converging to zero as $\e$ converges to zero. The random variables $\{X^\e_x\}$ are said to satisfy a Freidlin-Wentzell uniform large deviations principle with respect to the rate functions $I_x$ with speed $a(\e)$ uniformly over $\mathscr{A}$, if
  \begin{enumerate}[(a)]
    \item For any $A \in \mathscr{A}$, $s_0>0$, and $\delta>0$,
    \begin{equation} \label{eq:fwuldp-lower}
      \liminf_{\e \to 0} \inf_{x \in A} \inf_{\varphi \in \Phi_x(s_0)} \left(a(\e) \log \Pro(\rho(X^\e_x, \varphi)<\delta) + I_x(\varphi) \right) \geq 0.
    \end{equation}
    \item For any $A \in \mathscr{A}$, $s_0>0$, and $\delta>0$,
    \begin{equation} \label{eq:fwuldp-upper}
      \limsup_{\e \to 0} \sup_{x \in A} \sup_{s \in [0,s_0]} \left(a(\e) \log \Pro(\dist(X^\e_x, \Phi_x(s)) \geq \delta) + s \right) \leq 0.
    \end{equation}
  \end{enumerate}
\end{definition}

The next definition of uniform large deviations principle can be found in Corollary 5.6.15 of \cite{dz}. For the purposes of this paper, we will consider this as a definition.
For any set $G \subset \mathcal{E}$, let $I_x(G) := \inf_{\varphi \in G} I_x(\varphi)$.
\begin{definition}[ Dembo-Zeitouni uniform large deviations principle \\ (DZULDP) ] \label{def:dzuldp}
  Let $\mathscr{A}$ be a collection of subsets of $\mathcal{E}_0$ and $a(\e)$ be a function converging to zero as $\e$ converges to zero. The random variables $\{X^\e_x\}$ are said to satisfy a Dembo-Zeitouni uniform large deviations principle with respect to the rate functions $I_x$ with speed $a(\e)$ uniformly over $\mathscr{A}$, if
  \begin{enumerate}[(a)]
    \item For any $A \in \mathscr{A}$ and any open $G \subset \mathcal{E}$,
    \begin{equation} \label{eq:dzuldp-lower}
      \liminf_{\e \to 0} \inf_{x \in A}  \left(a(\e) \log \Pro(X^\e_x\in G)  \right) \geq -\sup_{x \in A} I_x(G).
    \end{equation}
    \item For any $A \in \mathscr{A}$ and any closed $F \subset \mathcal{E}$,
    \begin{equation} \label{eq:dzuldp-upper}
      \limsup_{\e \to 0} \sup_{x \in A}  \left(a(\e) \log \Pro(X^\e_x\in F)   \right) \leq -\inf_{x \in A} I_x(F).
    \end{equation}
  \end{enumerate}
\end{definition}

The third definition of a uniform large deviations principle is called the uniform Laplace principle. The uniform Laplace principle can be found in Definition 1.2.6 of \cite{de}. Based on the variational principle and the weak convergence approach in the papers by Budhiraja, Bou\'e, Dupuis, and Maroulas \cite{bd-1998,bd-2000,bdm-2008}, the uniform Laplace principle can be easier to verify directly than either of the uniform large deviations principles.
\begin{definition}[Uniform Laplace principle (ULP)] \label{def:ulp}
  Let $\mathscr{A}$ be a collection of subsets of $\mathcal{E}_0$ and $a(\e)$ be a function converging to zero as $\e$ converges to zero. The random variables $\{X^\e_x\}$ are said to satisfy a uniform Laplace principle with respect to the rate functions $I_x$ with speed $a(\e)$ uniformly over $\mathscr{A}$, if for any $A \in \mathscr{A}$ and any bounded, continuous $h: \mathcal{E} \to \mathbb{R}$,
  \begin{equation} \label{eq:ulp}
    \lim_{\e \to 0} \sup_{x \in A} \left| a(\e) \log \E \exp \left(-\frac{h(X^\e_x)}{a(\e)} \right) + \inf_{\varphi \in \mathcal{E}} \{h(\varphi) + I_x(\varphi)\} \right|=0.
  \end{equation}
\end{definition}

We now state the main assumptions and results of this paper.
\begin{assumption} \label{assum:compact-level-union} 
     $\mathscr{A}$ is a collection of subsets of $\mathcal{E}_0$ with the property that for any $s \geq0$ and $A \in \mathscr{A}$, $\bigcup_{x \in A} \Phi_x(s)$ is a pre-compact subset of $\mathcal{E}$.
\end{assumption}
\begin{theorem} \label{thm:FWULDP=ULP}
 Under Assumption \ref{assum:compact-level-union}, the FWULDP and ULP are equivalent.
\end{theorem}
Theorem \ref{thm:FWULDP=ULP} is proven in Section \ref{S:Proof-FWULDP=ULP}.
\vspace{.3cm}

The equivalence between the FWULDP and DZULDP requires extra topological assumptions on $\mathcal{E}_0$.
\begin{assumption}\label{assum:compact-initial-conds} \hspace{1cm}
  \begin{enumerate}[(a)]
   \item $\mathcal{E}_0$ is a Polish space with metric $\rho_0$.
   \item $\mathscr{A}$ is the collection of compact subsets of $\mathcal{E}_0$.
   \item For every $x \in \mathcal{E}_0$, $I_x$ is a good rate function.
   \item The level sets are continuous in the Hausdorff metric in the sense that for any $s\geq 0$,
  \[\lim_{n \to +\infty} \rho_0 (x_n,x)=0 \text{ implies }\lim_{n \to +\infty} \lambda(\Phi_{x_n}(s), \Phi_x(s))=0.\]
  \end{enumerate}
\end{assumption}
\begin{theorem} \label{thm:FWULDP=DZULDP}
  Under Assumption \ref{assum:compact-initial-conds}, the FWULDP and DZULDP are equivalent.
\end{theorem}
Theorem \ref{thm:FWULDP=DZULDP} is proven in Section \ref{S:Proof-FWULDP=DZULDP}.
\vspace{.3cm}

As we will show in the counterexamples (Section \ref{S:counterexamples}), the main reason that the ULP can fail if the FWULDP holds is that the test function $h: \mathcal{E} \to \mathbb{R}$ is continuous but not uniformly continuous. This observation inspires the introduction of the equicontinuous uniform Laplace principle (EULP).
A family $L\subset C(\mathcal{E})$ of functions from $\mathcal{E}$ to $\mathbb{R}$ is equibounded and equicontinuous if
\[\sup_{h \in L} \sup_{\varphi \in \mathcal{E}} |h(\varphi)| <+\infty \text{ and } \lim_{\delta \to 0} \sup_{h \in L}\sup_{\rho(\varphi, \psi)<\delta}|h(\varphi) - h(\psi)| = 0.\]
\begin{definition}[Equicontinuous uniform Laplace principle] \label{def:eulp}
  Let $\mathscr{A}$ be a collection of subsets of $\mathcal{E}_0$ and $a(\e)$ be a function converging to zero as $\e$ converges to zero. The random variables $\{X^\e_x\}$ are said to satisfy an equicontinuous uniform Laplace principle with respect to the rate functions $I_x$ with speed $a(\e)$ uniformly over $\mathscr{A}$, if for any $A \in \mathscr{A}$ and any collection $L\subset C(\mathcal{E})$ of equibounded and equicontinuous functions from $\mathcal{E}$ to $\mathbb{R}$,
  \begin{equation} \label{eq:eulp}
    \lim_{\e \to 0} \sup_{x \in A} \sup_{h \in L}  \left|a(\e) \log \E \exp \left(-\frac{h(X^\e_x)}{a(\e)} \right) + \inf_{\varphi \in \mathcal{E}} \{h(\varphi) + I_x(\varphi)\}  \right|=0.
  \end{equation}
\end{definition}

\begin{theorem} \label{thm:FWULDP=EULP}
  The EULP and the FWULDP are equivalent with no extra assumptions.
\end{theorem}
Theorem \ref{thm:FWULDP=EULP} is proven in Section \ref{S:Proof-FWULDP=EULP}.
\vspace{.3cm}

Now that we have established the equality of the EULP and the FWULDP, we present some sufficient conditions that imply the EULP when $X^\e_x$ can be written as measurable mappings of an infinite dimensional Wiener process. This setting is inspired by the weak convergence approach of Budhiraja, Dupuis, and Maroulas \cite{bdm-2008}, but requires some modifications when we require uniformity over subsets of $\mathcal{E}_0$ that are not compact.

Let $\beta = \{\beta_k(t)\}_{k=1}^\infty$ be a collection of i.i.d. one-dimensional Brownian motions on a filtered probability  $(\Omega,\mathcal{F}, \{\mathcal{F}_t\},\Pro)$. Define the space $\mathbb{R}^\infty$ to be the space of sequences of real numbers endowed with the metric of componentwise convergence. Fix some $T>0$ and let $C([0,T]:\mathbb{R}^\infty)$ be the space of continuous functions from $[0,T] \to \mathbb{R}^\infty$ endowed with the metric of uniform convergence in time. $\beta$ is $C([0,T]:\mathbb{R}^\infty)$-valued with probability one. Let $U \subset \mathbb{R}^\infty$ be the subspace
\[U = \left\{u = \{u_k\}_{k=1}^\infty \in \mathbb{R}^\infty: \sum_{k=1}^\infty u_k^2 < +\infty\right\}.\]
When endowed with the inner product $\left<u,v\right>_U := \sum_{k=1}^\infty u_kv_k$, $U$ is a separable Hilbert space. Let $L^2([0,T]:U)$ be the set of twice differentiable $U$-valued functions on $[0,T]$ endowed with the norm
\[|u|_{L^2([0,T]:U)}^2 = \int_0^T |u(s)|_U^2 ds.\]
Let $\mathcal{P}_2$ be the collection of $\mathcal{F}_t$-adapted $U$-valued processes $u(t)$ with the property that $\Pro(|u|_{L^2([0,T]:U)}<+\infty) = 1$.
Let\\ $\mathcal{S}^N = \{u \in L^2([0,T]:U): |u|_{L^2([0,T]:U)}^2\leq N\}.$ Let $\mathcal{P}_2^N$ be the collection of $\mathcal{F}_t$-adapted $U$-valued processes $u(t)$ such that $\Pro(u \in \mathcal{S}^N)=1.$

For $\e>0$ and $x \in \mathcal{E}_0$, let $\mathscr{G}^\e_x$ be measurable maps from $ C([0,T]:\mathbb{R}^\infty) \to \mathcal{E}$. In this section we establish a set of conditions on $\mathscr{G}^\e_x$ guaranteeing that $X^\e_x = \mathscr{G}^\e_x(\sqrt{\e} \beta)$ satisfies an EULP.

\begin{assumption} \label{assum:mathcal-G}
  Assume that for any $x \in \mathcal{E}_0$, there exists a measurable mapping $\mathscr{G}^0_x: C([0,T]:\mathbb{R}^\infty) \to \mathcal{E}$ and a collection $\mathscr{A}$ of subsets of $\mathcal{E}_0$ such that for any $A \in \mathscr{A}$, $N>0$ and $\delta>0$,
    \[\lim_{\e \to 0} \sup_{x \in A}\sup_{u \in \mathcal{P}_2^N}\Pro\left( \rho \left(\mathscr{G}_x^\e\left( \sqrt{\e}\beta(\cdot) + \int_0^\cdot u(s)ds \right) , \mathscr{G}^0_x \left( \int_0^\cdot u(s)ds \right) \right)>\delta\right)=0.\]
\end{assumption}
Define the rate functions  $I_x: \mathcal{E} \to \mathbb{R}$ for $x \in \mathcal{E}_0$ by
\begin{equation} \label{eq:rate-fct}
  I_x(\varphi) = \inf \left\{\frac{1}{2}\int_0^T |u(s)|_U^2 ds : \varphi = \mathscr{G}^0_x\left( \int_0^\cdot u(s)ds \right) \right\}.
\end{equation}
The infimum is taken over all $u$ in $L^2([0,T]:U)$. We use the convention that the infimum of the empty set is $+\infty$.
\begin{remark} \label{rem:good}
If for fixed $x \in \mathcal{E}_0$ and any $N>0$, the level set
    \[\Phi_x(N)=\left\{\varphi \in \mathcal{E}: I_x(\varphi) \leq N \right\}=\left\{\mathscr{G}^0_x\left( \int_0^\cdot u(s)ds\right) : u \in \mathcal{S}^{2N} \right\}\]
    is a compact subset of $\mathcal{E}$, then $I_x$ is a good rate function.
\end{remark}
\begin{theorem} \label{thm:eulp-sufficient}
  Under Assumption \ref{assum:mathcal-G}, the $\mathcal{E}$-valued random variables $X^\e_x = \mathscr{G}^\e_x(\sqrt{\e}\beta)$ satisfy an EULP with respect to the rate function $I_x$ with speed $a(\e)=\e$ uniformly over $\mathscr{A}$.
\end{theorem}
The proof is presented in Section \ref{S:EULP}.
\vspace{.3cm}

The main difference between the weak convergence approach of \cite{bdm-2008} and Assumption \ref{assum:mathcal-G} is that the weak convergence approach requires that the mapping $(\e, x, u) \mapsto \mathscr{G}^\e_x\left( \sqrt{\e}\beta + \int_0^\cdot u(s)ds \right)$ be jointly continuous in an appropriate topology and that the $x$ belong to a compact set. When these continuity and compactness conditions are met, Assumption \ref{assum:mathcal-G} will follow. The EULP approach, on the other hand, does not require any continuity in $x$ or $u$. Instead, we merely require that the convergence of $\mathscr{G}^\e_x$ to $\mathscr{G}^0_x$ in probability must be uniform with respect to $x$ and $u$. In Section \ref{S:example} we show how this theory can be applied to prove that a family of Hilbert space valued stochastic equations exposed to small multiplicative noise satisfies an EULP that is uniform over initial conditions in bounded subsets of an infinite dimensional Hilbert space. The weak convergence approach cannot be used for such an example because bounded subsets of infinite dimensional Hilbert spaces are not generally compact.

\section{Counterexamples} \label{S:counterexamples}
Before proving the main results of the paper, we illustrate why Assumptions \ref{assum:compact-level-union} and \ref{assum:compact-initial-conds} are needed for Theorems \ref{thm:FWULDP=ULP} and \ref{thm:FWULDP=DZULDP} to hold. Using a simple example, we can demonstrate the FWULDP is not equivalent to the DZULDP or ULP.

The first counterexample is the simplest possible small noise equation $X^\e_x(t) := x + \sqrt{\e} W(t)$, where $W(t)$ is a one-dimensional Brownian motion and the initial condition $x \in \mathbb{R}$. For any $T>0$ let $C([0,T])$ be the space of continuous functions from $[0,T] \to \mathbb{R}$. We will consider the trajectories of $X^\e_x$ as $\mathcal{E} = C([0,T])$--valued random variables.  Let $|\varphi|_{C([0,T])} = \sup_{t \in [0,T]} |\varphi(t)|$ denote the supremum norm. For any $\varphi \in C([0,T])$ and $B \subset C([0,T])$, let $\dist(\varphi, B) = \inf_{\psi \in B} |\varphi - \psi|_{C([0,T])}$.

It is standard that for any $T>0$, the processes $\{\sqrt{\e}W(\cdot): \e>0\}$ satisfy a large deviations principle in $C([0,T])$ with rate function $ I_0: C([0,T]) \to \mathbb{R}$ given by
\[ I_0(\varphi) = \inf\left\{\frac{1}{2}\int_0^T |u(s)|^2ds: \varphi(t) = \int_0^t u(s)ds \right\}.\]
The infimum is taken over all $u \in L^2([0,T])$ and $I_0(\varphi) = +\infty$ if $\varphi$ cannot be written as $\varphi(t) = \int_0^tu(s)ds$ (meaning $\varphi$ is not absolutely continuous). Let $\Phi_0(s) = \{\varphi \in \mathcal{E}: I_0(\varphi) \leq s\}$ be the level sets of $I_0$.  We state this result without proof in the next theorem.

\begin{theorem}[Theorems 3.2.1-3.2.2 of \cite{fw}] \label{thm:W-LDP}
  For any fixed $T>0$, $\{\sqrt{\e} W\}_{\e>0}$ satisfies a large deviations principle with respect to the rate function $I_0$ with speed $a(\e) = \e$. In particular,
  \begin{enumerate}
    \item For any $\delta>0$ and $s_0>0$,
    \begin{equation} \label{eq:eW-lower}
      \liminf_{\e\to 0} \inf_{\varphi \in \Phi_0(s_0)} \left( \e \log \Pro(|\sqrt\e W - \varphi|_{C([0,T])} < \delta) + I_0(\varphi)\right) \geq 0.
    \end{equation}
    \item For any $\delta>0$ and $s_0>0$
    \begin{equation} \label{eq:eW-upper}
      \limsup_{\e\to 0} \sup_{s \in [0, s_0]} \left(\e \log \Pro(\dist(\sqrt\e W, \Phi_0(s))\geq \delta) + s \right) \leq 0.
    \end{equation}

  \end{enumerate}
\end{theorem}

The next theorem shows that the processes $X^\e_x(t)$ satisfy a uniform large deviations principle that is uniform over any measurable subset of $\mathbb{R}$ with respect to the rate function
\begin{equation} \label{eq:BM-rate-fct}
  I_x(\varphi) = \inf \left\{\frac{1}{2}\int_0^T |u(s)|^2 ds: \varphi(t) = x + \int_0^t u(s)ds \right\}.
\end{equation}
Let $\Phi_x(s) = \{\varphi \in C([0,T]): I_x(\varphi) \leq s\}$.
\begin{theorem} \label{thm:BM-FWULDP}
  Let $\mathscr{A}$ be the collection of all subsets of $\mathbb{R}$. Let $T>0$ The process $\{X^\e_x\}$ satisfies a FWULDP in $\mathcal{E} = C([0,T])$ with respect to the good rate functions $I_x$
  and speed $a(\e) = \e$ uniformly over $\mathscr{A}$. That is
  \begin{enumerate}
    \item For any $A \in \mathscr{A}$, $\delta>0$, and $s_0>0$,
    \begin{equation} \label{eq:BM-ULDP-lower}
      \liminf_{\e \to 0} \inf_{x \in A} \inf_{\varphi \in \Phi_x(s_0)} \left( \e \log \Pro(|X^\e_x - \varphi|_{C([0,T])} <\delta) + I_x(\varphi) \right) \geq 0.
    \end{equation}
    \item For any $A \in \mathscr{A}$, $\delta>0$ and $s_0>0$,
    \begin{equation} \label{eq:BM-ULDP-upper}
      \limsup_{\e \to 0} \sup_{x \in A} \sup_{s \in [0,s_0]} \left( \e \log \Pro(\dist(X^\e_x, \Phi_x(s))\geq\delta) + s \right) \leq 0.
    \end{equation}
  \end{enumerate}
\end{theorem}

\begin{proof}
  It is sufficient to prove this theorem with $A= \mathbb{R}$. If \eqref{eq:BM-ULDP-lower} and \eqref{eq:BM-ULDP-upper} hold with $A = \mathbb{R}$, then they also hold for any subset of $\mathbb{R}$. Fix $s_0>0$.
  For any $x\in\RR$, the elements of $\Phi_x(s_0)$ are translations of elements of $\Phi_0(s_0)$. In particular, for any $\varphi \in \Phi_x(s_0)$, $\psi(t) := \varphi(t) - x$ is in $\Phi_0(s_0)$ and $I_x(\varphi) = I_0(\psi)$. Similarly, $X^\e_x(t) - x = \sqrt{\e}W(t)$. Therefore, for any $\varphi \in \Phi_x(s_0)$, and $\psi(t) = \varphi(t) - x$,
  it follows that $|X^\e_x - \varphi|_{C([0,T])} = |\sqrt{\e}W - \psi|_{C([0,T])}$.
  Therefore,
  \begin{align*}
     &\liminf_{\e \to 0} \inf_{x \in \RR} \inf_{\varphi \in \Phi_x(s_0)} \left(\e \log \Pro(|X^{\e}_{x} - \varphi|_{C([0,T])} < \delta) + I_x(\varphi) \right)\\
     &\geq \liminf_{\e \to 0}  \inf_{\psi \in \Phi_0(s_0)} \left( \e \log \Pro(\sqrt{\e}W - \psi|_{C([0,T])} < \delta) + I_0 (\psi)\right)\\
     &\geq 0
  \end{align*}
  where the last line follows from \eqref{eq:eW-lower}. Therefore, the FWULDP lower bound \eqref{eq:BM-ULDP-lower} holds.

  The upper bound is similar. Because of the definitions of $X^\e_x$ and the rate functions,
  \[\left\{\dist(X^\e_x, \Phi_x(s)) \geq \delta \right\} = \left\{ \dist(\sqrt{\e} W, \Phi_0(s))\geq \delta \right\}.\]
  Therefore
  \begin{align*}
    &\limsup_{\e \to 0} \sup_{x \in \mathbb{R}} \sup_{s \in[0, s_0]} \left(\e \log \Pro(\dist(X^\e_x, \Phi_x(s)) \geq \delta) + s \right)\\
    &= \limsup_{\e \to 0} \sup_{s \in[0, s_0]} \left(\e \log \Pro(\dist(\sqrt{\e} W,  \Phi_0(s))\geq \delta) + s \right)\\
    &\leq 0.
  \end{align*}
  The last line follows from \eqref{eq:eW-upper}.
\end{proof}

While the process $X^\e_x(t) = x + \sqrt{\e}W(t)$ satisfy a FWULDP in $C([0,T])$ that is uniform over initial conditions in all of $\mathbb{R}$, they do not satisfy a DZULDP or ULP over all initial conditions in $\mathbb{R}$. It is clear that Assumption \ref{assum:compact-initial-conds} cannot  hold when the set $A$ of initial conditions is not a compact set and Assumption \ref{assum:compact-level-union} cannot hold when the set $A$ of initial conditions is not a precompact set.
\begin{theorem} \label{thm:counter-BM-DZ}
  The process $X^\e_x$ does not satisfy a DZULDP when $\mathscr{A}$ contains all subsets of $\mathbb{R}$.
\end{theorem}

\begin{proof}
  We demonstrate that neither the lower bound \eqref{eq:dzuldp-lower} nor the upper bound \eqref{eq:dzuldp-upper} are satisfied over unbounded sets. Let $A = \NN$. For any $n \in A$, let $f_n(t) = n + t$. There is nothing special about $f_n$. The proof could use any set of functions that are just translated by initial condition. Define the open set
  \[G = \bigcup_{n \in A} \{\varphi \in C([0,T]): |\varphi - f_n|_{C([0,T])} < 2^{-n}\}.\]
  $G$ is an open set because it is the union of open sets.
  Unfortunately, because $X^\e_n(t) - f_n(t) = \sqrt{\e}W(t) - t$, for every $\e>0$,
  \begin{align*}
    &\inf_{n \in A} \Pro(X^\e_n \in G) = \inf_{n \in A} \Pro\left(\sup_{t \in [0,T]} |\sqrt{\e} W(t) -t| < 2^{-n}\right) \\
    &= \Pro\left(\sup_{t \in [0,T]} |\sqrt{\e}W(t) - t|=0\right)=0.
  \end{align*}
  It follows that for every $\e>0$
  \begin{align*}
    &\liminf_{\e \to 0} \inf_{n \in A}  \left(\e \log \Pro (X^\e_n \in G) \right) = -\infty,
  \end{align*}
  while
  $$\sup_{n\in A} I_n(G) \leq \sup_n I_n(f_n) = \int_0^T 1^2 dt = \frac{T}{2}.$$
  This analysis shows that $X^\e_x(t) = x + \sqrt{\e}W(t)$ does not satisfy \eqref{eq:dzuldp-lower}

  For the upper bound, consider the closed set
  \[F = \bigcup_{n \in A} \left\{\varphi \in C([0,T]): \varphi(0) = n \text{ and }  \dist(\varphi, \Phi_n(1))\geq 2^{-n}  \right\}.\]
  This is closed because it is a union of disjoint closed sets each of which is at least distance 1 from the others (because the initial conditions are at least distance 1 from each other). Because $\dist(X^\e_x, \Phi_x(1)) = \dist(\sqrt{\e}W, \Phi_0(1))$, it follows that
  \begin{align*}
    &\sup_{n \in A} \Pro(X^\e_n \in F) = \sup_{n \in A} \Pro(\dist(\sqrt{\e} W,  \Phi_0(1)) \geq 2^{-n}) \\
    &= \Pro(\dist(\sqrt{\e}W,  \Phi_0(1)) >0) = 1.
  \end{align*}
  The above is equal to $1$ because $\Phi_0(1)$ contains only differentiable functions and $\sqrt{\e}W$ has rough paths, so $\Pro(\sqrt{\e} W \not \in \Phi_0(1)) =1$.
  Then
  $$\sup_{n \in A} \e \log \Pro(X^\e_n \in F) = 0 \text{ and } \inf_{n \in A} I_n(F) \geq 1,$$ so the upper bound \eqref{eq:dzuldp-upper} cannot be true.
\end{proof}

\begin{remark} \label{rem:counterexample-bounded}
  An unbounded set is not even required for the above counterexample proof. The bounded set $A = \{2^{-n}\}_{n=1}^\infty$ with the open set equal to the disjoint union of open balls
  \begin{equation} \label{eq:counterexample-open-set}
    G = \bigcup_{n=1}^\infty \left\{\varphi \in C([0,T]): \sup_{t \in [0,T]}|\varphi(t) - (2^{-n} +t)|<4^{-n}\right\}
  \end{equation}
  is sufficient to prove that \eqref{eq:dzuldp-lower} does not hold via the same arguments as the proof of Theorem \ref{thm:counter-BM-DZ}. For any $x \in A$, $I_x(G)\leq \int_0^T 1^2 dt =\frac{T}{2}$, but because of the degeneracy of the open balls, $\inf_{x \in A} \Pro(X^\e_x \in G) = 0$. Therefore, the DZULDP lower bound \eqref{eq:dzuldp-lower} cannot hold. Theorem \ref{thm:FWULDP=DZULDP} truly requires compactness of the initial conditions. $A$ is pre-compact but not compact.
  Note that \eqref{eq:dzuldp-lower} does hold if $A$ is the compact set $\{2^{-n}\}_{n=1}^\infty \cup\{0\}$ because $I_0(G) = +\infty$.
\end{remark}

\begin{remark} \label{rem:counterexample-no-Hausdorff-cont}
  Even if $A$ is compact, we can build a counterexample to the DZULDP if the mapping $x \mapsto \Phi_x(s)$ is not continuous in the Hausdorff metric as in Assumption \ref{assum:compact-initial-conds}. Consider the family of processes $Z^\e_x(t) = X^\e_x(t) = x + \sqrt{\e}W(t)$ if $x \not =0$ and $Z^\e_0(t) = X^\e_{1/2}(t) = \frac{1}{2} + \sqrt{\e}W(t)$. Let $A = \{2^{-n}\}_{n=1}^\infty \cup \{0\}$, which is a compact set. Let $G$ be as in \eqref{eq:counterexample-open-set}.

  The rate function for $Z$ is $\tilde I_x = I_x$ for $x \not = 0$ and $\tilde I_0 = I_{1/2}$. Let $\tilde\Phi_x = \{\varphi \in \mathcal{E}: \tilde I_x (\varphi) \leq s\}.$
  In this case, for any $x \in A$, $\tilde I_x(G) \leq \frac{T}{2}$ but $\inf_{x \in A} \Pro(Z^\e_x \in G) = 0$.

  The fact that $A$ is compact does not help because the map $x \mapsto \tilde \Phi_x(s)$ is discontinuous at $0$ in the Hausdorff metric.
\end{remark}

The counterexample for the FWULDP--ULP equivalence is pretty much the same. We require a test function $h: C([0,T]) \to \mathbb{R}$ that is bounded and continuous, but not uniformly continuous.

\begin{theorem} \label{thm:ulp-counter}
  The processes $\{X^\e_x\}$ do not satisfy a ULP uniformly over all subsets of $\mathbb{R}$.
\end{theorem}

\begin{proof}
Let $A= \mathbb{N}$. For $n \in A$, let $f_n(t) = n+ t$. Let $j> \frac{T}{2}$. Define $h: C([0,T]) \to \mathbb{R}$ by
\[h(\varphi) = j \min\left\{1,\min_{n \in A}\{ 2^n|\varphi - f_n|_{C([0,T])}\}\right\}.\]
This function has the properties that $h(f_n) = 0$ for all $n \in A$ and $h(\varphi) = j$ if $|\varphi - f_n|_{C([0,T])}\geq 2^{-n}$ for all $n\in A$. Otherwise $h(\varphi) \in [0,j]$.

For any $n \in A$,
\[\inf_{\varphi \in C([0,T])} \{h(\varphi) + I_n(\varphi)\} \leq h(f_n) + I_n(f_n) = \frac{1}{2}\int_0^T 1^2dt = \frac{T}{2}.\]

Because $X^\e_n(0)=n$, $-h(X^\e_n)\leq -j$ when $|X^\e_n -f_n|_{C([0,T])}\geq  2^{-n}$ and because $h\geq 0$, $-h(X^\e_n)\leq 0$ when $|X^\e_n  -f_n|\leq 2^{-n}$. Therefore,
\begin{align*}
  &\inf_{n\in A} \E \exp \left(-\frac{h(X_n^\e)}{\e} \right) \leq e^{-\frac{j}{\e}} +   \inf_{n\in A}\Pro(|X^\e_n - f_n|<2^{-n}) \\
  &\leq  e^{-\frac{j}{\e}} +  \inf_{n \in A}\Pro\left(\sup_{t \in [0,T]}|\sqrt{\e}W(t) - t|<2^{-n} \right)
  \leq e^{-\frac{j}{\e}}.
\end{align*}
Therefore,
\[\inf_{n \in A} \left( \e \log \E \exp \left(-\frac{h(X_n^\e)}{\e} \right) + \inf_{\varphi \in C([0,T])}\{h(\varphi) +  I_n(\varphi)\} \right) \leq -j + \frac{T}{2}<0.\]
$X^\e_n$ does not satisfy a ULP over $A$ because \eqref{eq:ulp} fails.
\end{proof}

$X^\e_x$ fails to satisfy a ULP because of a lack of uniform continuity.
The test function $h: C([0,T]) \to \mathbb{R}$ in the proof of Theorem \ref{thm:ulp-counter} is continuous but not uniformly continuous. 
This counterexample inspires the formulation of the EULP in Definition \ref{def:eulp}

The set $G$ in the proof of Theorem \ref{thm:counter-BM-DZ} is open, but it is open in a very degenerate way. The open set is a union of $C([0,T])$-balls of arbitrarily small radii. To use imprecise language: such a $G$ is open, but it is not uniformly open.

A generalization on the DZULDP was introduced by Lipshutz \cite{l-2017} to exclude testing on sets like those in the proof of Theorem \ref{thm:counter-BM-DZ}. For any open set $G \subset \mathcal{E}$, let
\begin{equation} \label{eq:G-eta}
  G_\eta = \{\varphi \in G: \dist(\varphi, \mathcal{E} \setminus G) > \eta\},
\end{equation}
and for any closed set $F\subset \mathcal{E}$ let
\begin{equation} \label{eq:F-eta}
  F^\eta = \{\varphi \in \mathcal{E}: \dist(\varphi, F) \leq \eta\}.
\end{equation}

\begin{definition}[Lipshutz uniform large deviations principle (LULDP)]  \label{def:luldp}
  Let $\mathscr{A}$ be a collection of subsets of $\mathcal{E}_0$ and $a(\e)$ be a function converging to zero as $\e$ converges to zero. The random variables $\{X^\e_x\}$ are said to satisfy a uniform large deviations principle with respect to the rate functions $I_x$ with speed $a(\e)$ uniformly over $\mathscr{A}$, if
  \begin{enumerate}[(a)]
    \item For any $A \in \mathscr{A}$ and $G \subset \mathcal{E}$ open,
    \begin{equation} \label{eq:luldp-lower}
      \liminf_{\e \to 0} \inf_{x \in A}  \left(a(\e) \log \Pro(X^\e_x\in G)  \right) \geq -\lim_{\eta \to 0}\sup_{x \in A} I_x(G_\eta).
    \end{equation}
    \item For any $A \in \mathscr{A}$ and $F \subset \mathcal{E}$ closed, and $s \leq I_x(F)$,
    \begin{equation} \label{eq:luldp-upper}
      \limsup_{\e \to 0} \sup_{x \in A}  \left(a(\e) \log \Pro(X^\e_x\in F)   \right) \leq -\lim_{\eta \to 0}\inf_{x \in A} I_x(F^\eta).
    \end{equation}
  \end{enumerate}
\end{definition}
This definition enables the DZULDP to be used over non-compact sets, but it is not equivalent to the FWULDP.
We give an example where the FWULDP (Definition \ref{def:fwuldp}) is not satisfied but the LULDP (Definition \ref{def:luldp} is satisfied).
Consider the process for $\e>0$ and $x \in \mathbb{R}$
\begin{equation} \label{eq:counterex2}
  Y^\e_x(t) := (1+\e)x + \sqrt{\e}W(t) = X^\e_{(1+\e)x}.
\end{equation}
Let $I_x$ be the same rate function defined in \eqref{eq:BM-rate-fct}. It is not difficult to show that $Y^\e_x$ satisfies a FWULDP in $C([0,T])$ with respect to $I_x$ that is uniform with respect to initial conditions $x$ in bounded subsets of $\mathbb{R}$. We will show that $Y^\e_x$ does not satisfy a FWULDP over initial conditions in unbounded sets. On the other hand, $Y^\e_x$ does satisfy a LULDP over the whole space.

\begin{theorem}
  $\{Y^\e_x\}$ does not satisfy a FWULDP over $x \in \mathbb{R}$ with respect to its rate function $I_x$.
\end{theorem}

\begin{proof}
  If $\varphi \in \Phi_x(s)$, then $\varphi(0) = x$. It follows that
  \[|Y^\e_x - \varphi|_{C([0,T])} \geq |Y^\e_x(0) - \varphi(0)| = |(1+\e)x - x| = \e|x|.\]
  For any $\delta>0$ and $\e>0$, there exists $x \in \mathbb{R}$ such that $\e|x|>\delta$.
  Therefore,
  \[\inf_{x \in \mathbb{R}}\inf_{\varphi \in \Phi_x(s)}\Pro(|Y^\e_x - \varphi|_{C([0,T])}<\delta) = 0\]
  and
  \[\inf_{x \in \mathbb{R}} \inf_{\varphi \in \Phi_x(s)} \e \log \Pro(|Y^\e_x - \varphi|_{C([0,T])}<\delta) = -\infty.\]
  This proves that \eqref{eq:fwuldp-lower} fails.
  Along the same lines,
  \[\sup_{x \in \mathbb{R}} \sup_{s \in [0, s_0]} \e \log \Pro(\dist(Y^\e_x, \Phi_x(s))\geq \delta) = 0\]
  and \eqref{eq:fwuldp-upper} fails.
\end{proof}

Despite the fact that $Y^\e_x$ does not satisfy a FWULDP over $x \in \mathbb{R}$, it does satisfy a LULDP.
\begin{theorem} \label{thm:luldp-counter}
  For any $G\subset C([0,T])$ open and any $\eta>0$, let $G_\eta$ be as in \eqref{eq:G-eta}, then
  \begin{equation} \label{eq:luldp-lower-true}
    \liminf_{\e \to 0} \inf_{x \in \mathbb{R}} \Pro(Y^\e_x \in G) \geq -\sup_{x \in \RR} I_x(G_\eta).
  \end{equation}
  For any $F \subset C([0,T])$ closed and any $\eta>0$, let $F^\eta$ be as in \eqref{eq:F-eta}, then
  \begin{equation} \label{eq:luldp-upper-true}
    \limsup_{\e \to 0} \sup_{x \in \mathbb{R}} \Pro(Y^\e_x \in F) \leq \inf_{x \in \RR}  I_x(F^\eta).
  \end{equation}
\end{theorem}
\begin{proof}
  Fix $\eta>0$. If $\sup_{x \in \RR} I_x(G_\eta)=+\infty$, then \eqref{eq:luldp-lower-true} is trivially true. Assume that $\sup_{x \in \RR} I_x(G_\eta)=:s_0<+\infty$.
  This means that for any $s>s_0$ and $x \in \mathbb{R}$, there exists $\varphi_x \in G_\eta$ such that $I_x(\varphi_x)\leq s$. Because $\varphi_x \in G_\eta$, the $\eta$-open balls
  \[\{\varphi \in C([0,T]): |\varphi - \varphi_x|_{C([0,T])}<\eta\} \subset G.\]
  For any $x \in \mathbb{R}$ and $\e>0$,
  \begin{align*}
    &\Pro(Y^\e_x \in G) = \Pro(X^\e_{(1+\e)x} \in G) \geq \Pro(|X^\e_{(1+\e)x} -\varphi_{(1+\e)x}|_{C([0,T])}<\eta).
  \end{align*}
  By \eqref{eq:BM-ULDP-lower},
  \begin{align*}
    &\liminf_{\e \to 0} \inf_{x \in \mathbb{R}} \Pro(Y^\e_x \in G)\\
     &\geq \liminf_{\e \to 0}  \inf_{x \in \mathbb{R}}\inf_{\varphi \in \Phi_{(1+\e)x}(s)} \Pro(|X^\e_{(1+\e)x} -\varphi|_{C([0,T])}<\eta) \\
    &\geq \liminf_{\e \to 0} \inf_{x \in \mathbb{R}} \inf_{\varphi \in \Phi_x(s)} \Pro(|X^\e_x - \varphi|_{C([0,T])}<\eta) \\
    &\geq -s.
  \end{align*}
  Recall that $s>s_0$ was arbitrary so it follows that
  \[\liminf_{\e \to 0} \inf_{x \in \mathbb{R}} \Pro(Y^\e_x \in G) \geq -s_0 = \sup_{x \in \RR} I_x(G_\eta)\]
   which proves \eqref{eq:luldp-lower-true}.

  The upper bound \eqref{eq:luldp-upper-true} is trivially true if $\inf_{x \in \mathbb{R}} I_x(F^\eta) = 0$. Assume that $\inf_{x \in \mathbb{R}} I_x(F^\eta)= :s_0>0$. A consequence is that for any $x \in \mathbb{R}$ and $s <s_0$, $F^\eta \cap \Phi_x(s) = \emptyset$. By the definition of $F^\eta$, it follows that for any $s <s_0$ and $x \in \mathbb{R}$,
  \[F \subset \{\varphi \in C([0,T]): \dist(\varphi, \Phi_x(s))\geq \eta\}.\]
  Recalling that $Y^\e_x = X^\e_{(1+\e)x}$, we see that.
  \[\Pro(Y^\e_x \in F) = \Pro(X^\e_{(1+\e)x} \in F) \leq \Pro(\dist(X^\e_{(1+\e)x}, \Phi_{(1+\e)x}(s)) \geq \eta).\]
  Then by \eqref{eq:BM-ULDP-upper},
  \[\limsup_{\e \to 0} \sup_{x \in \mathbb{R}} \e \log \Pro(Y^\e_x \in F) \leq -s. \]
  The choice of $s \leq s_0$ was arbitrary, so
  \[\limsup_{\e \to 0} \sup_{x \in \mathbb{R}} \e \log \Pro(Y^\e_x \in F) \leq -s_0.\]
  Because of our choice of $s_0$, this proves \eqref{eq:luldp-upper-true}.
\end{proof}

This $Y^\e_x(t) = (1+\e)x + \sqrt{\e}W(t)$ example illustrates the important difference between the FWULDP and the LULDP.
In the FWULDP, (see \eqref{eq:fwuldp-lower}), the probability divergence rate $a(\e) \log \Pro(\rho(Y^\e_x, \varphi)<\delta)$ is always compared to $I_x(\varphi)$. In \eqref{eq:luldp-lower}, $a(\e) \log \Pro(Y^\e_x \in G)$ is compared to $\sup_{x \in A}I_x(G)$. In the proof of Theorem \ref{thm:luldp-counter}, this allowed us to compare $\e \log \Pro(Y^\e_x \in G)$ to $I_{(1+\e)x}(G)$. The FWULDP insists that exponential decay of probabilities about $X^\e_x$ are described by $I_x$, but the LULDP allows us to describe the decay of these probabilities with $I_y$ for $y \not = x$.

\section{Example - Hilbert space valued process} \label{S:example}
The EULP will be most useful for studying large deviations principles for infinite dimensional systems.  Let $H$ be an infinite dimensional separable Hilbert space. Let $C([0,T]:H)$ be the Banach space of continuous functions from $[0,T] \to H$ endowed with the norm
\[|\varphi|_{C([0,T]:H)} = \sup_{t \in [0,T]} |\varphi(t)|_H.\]
We will show under very general assumptions that an $H$-valued  family of stochastic processes satisfies a FWULDP uniformly over bounded sets of initial conditions in the Hilbert space. If we assume that the multiplicative noise coefficient is bounded in an appropriate sense then the FWULDP will be uniform over initial conditions in the entire space. These results show that there is no reason to restrict the study of uniform large deviations principles to compact sets of initial condition.

We consider the following small noise $H$-valued stochastic equation with Lipschitz continuous coefficients. See Chapter 7.1.1 of \cite{dpz} for more information about such a system.
\begin{equation} \label{eq:SPDE}
  dX^\e_x(t) = [\mathcal{A}X^\e_x(t) + B(X^\e_x(t))]dt + \sqrt{\e}G(X^\e_x(t))dw(t), \ \ \ X^\e_x(0) = x.
\end{equation}
In the above equation, $X^\e_x(t)$ and $x$  are $H$-valued. $\mathcal{A}: D(\mathcal{A}) \subset H \to H$ is an unbounded linear operator that generates a $C_0$ semigroup on $H$ called $S(t)$.
The mild solution to \eqref{eq:SPDE} is defined to be the $\mathcal{F}_t$-adapted $C([0,T]:H)$ solution to the integral equation
\begin{equation} \label{eq:mild}
  X^\e_x(t) = S(t)x + \int_0^t S(t-s)B(X^\e_x(s))ds + \sqrt{\e} \int_0^t S(t-s)G(X^\e_x(s))dw(s).
\end{equation}

The noise $w(t)$ is a cylindrical Wiener process. Let $\mathbb{R}^\infty$ be the collection of sequences of real numbers endowed with the metric of componentwise convergence. Let $w(t) = \{\beta_k(t)\}_{k=1}^\infty$ be a family of i.i.d. one-dimensional Brownian motions on a filtered probability space $(\Omega, \mathcal{F}, \{\mathcal{F}_t\}, \Pro)$. Define the Hilbert space $U = \{u = \{u_k\}_{k=1}^\infty \in \mathbb{R}^\infty: \sum_{k=1}^\infty u_k^2 < +\infty\}$ endowed with the inner product $\left<u,v\right>_U = \sum_{k=1}^\infty u_kv_k$.

Let $L_2 :=L_2(U,H)$ denote the space of Hilbert-Schmidt operators from $U$ to $H$. The Hilbert-Schmidt norm of a bounded linear operator $M: U \to H$ is
\begin{equation}
  \|M\|_{L_2}^2 = \sum_{k=1}^\infty |Me_k|_H^2
\end{equation}
where $\{e_k\}$ is any complete any orthonormal basis of $U$.

We assume that for any $t>0$ and $x \in H$, $S(t)G(x)$ is a Hilbert-Schmidt operator from $U$ to $H$.

The next assumptions describe that both $B$ and $G$ are Lipschitz continuous in an appropriate sense.
\begin{assumption} \label{assum:B-Lip}
  The nonlinear operator $B:H \to H$ is Lipschitz continuous. There exists a constant $\kappa>0$ such that
  \begin{enumerate}[(a)]
    \item For any $x,y \in H$,
    \begin{equation}
      |B(x) - B(y)|_H  \leq \kappa |x-y|_H.
    \end{equation}
    \item For any $x \in H$,
    \begin{equation}
       |B(x)|\leq \kappa (1 + |x|_H).
    \end{equation}
  \end{enumerate}
\end{assumption}
\begin{assumption} \label{assum:G-Lip}
  There exists a locally square integrable mapping $K: [0,+\infty) \to [0,+\infty)$ and a constant $\alpha \in (0,1/2)$ such that for any $t>0$
  \begin{equation} \label{eq:K-int}
    \int_0^t s^{-2\alpha} K^2(s) ds < +\infty
  \end{equation}
  such that
  \begin{enumerate}[(a)]
    \item For any $x,y \in H$ and $t>0$,
    \begin{equation}
      \|S(t)G(x) - S(t)G(y)\|_{L_2} \leq K(t) |x-y|_H.
    \end{equation}
    \item We either assume
    \begin{enumerate}[(i)]
      \item $G$ is bounded in the sense that for any $x \in H$  and $t>0$
      \begin{equation} \label{eq:G-bounded}
        \|S(t)G(x)\|_{L_2} \leq K(t)
      \end{equation}
      or
      \item $G$ has linear growth in the sense that
      \begin{equation} \label{eq:G-grows}
        \|S(t)G(x)\|_{L_2} \leq K(t)(1 + |x|_H).
      \end{equation}
    \end{enumerate}
  \end{enumerate}
\end{assumption}

Under these assumptions, the $C([0,T]:H)$ solution of \eqref{eq:mild} exists and is unique for any $x \in H$ and $\e>0$ \cite[Theorem 7.5]{dpz}. Furthermore, for every $x \in H$ there exists a measurable map $\mathscr{G}_x:  C([0,T]:\mathbb{R}^\infty) \to C([0,T]:H)$ such that for any $x \in H$ and $\e>0$, $X^{\e}_x = \mathscr{G}_x( \sqrt{\e}w)$.

Define the space $L^2([0,T]:U)$ to be the space of $U$ valued processes such that
$|u|_{L^2([0,T]:U)}^2: = \int_0^T |u(s)|_U^2 ds < +\infty$. Let $\mathcal{P}_2$ be the collection of $\mathcal{F}_t$-adapted $U$-valued controls $u(t)$ such that $\Pro(|u|_{L^2([0,T]:U)} < +\infty)=1$. Let $\mathcal{S}^N = \{u \in L^2([0,T]:U): |u|_{L^2([0,T]:U)}^2 \leq N\}$. Let $\mathcal{P}_2^N= \{u \in \mathcal{P}_2: \Pro(u \in \mathcal{S}^N)=1\}$. 

For any $u \in \mathcal{P}_2$ and $\e\geq0$, let $X^{\e,u}_x := \mathscr{G}_x\left( \sqrt{\e}W + \int_0^\cdot u(s)ds \right)$.
Such a process solves
\begin{align} \label{eq:control-mild}
  X^{\e,u}_x(t) = &S(t)x + \int_0^t S(t-s)B(X^{\e,u}_x(s))ds \nonumber \\
  &+ \sqrt{\e}\int_0^t S(t-s)G(X^{\e,u}_x(s))dw(s) \nonumber\\
  &+ \int_0^t S(t-s)G(X^{\e,u}_x(s))u(s)ds.
\end{align}
We will prove in Lemma \ref{lem:a-priori} that $X^{\e,u}_x$ is well-posed for any $x \in H$, $\e>0$, $N>0$, and $u \in \mathcal{P}_2^N$.

For any $x \in H$ define the rate function $I_x: C([0,T]:H) \to [0,+\infty]$ by
\begin{equation}
  I_x(\varphi) = \inf \left\{\frac{1}{2} \int_0^T |u(s)|_U^2 ds : \varphi = X^{0,u}_x \right\}.
\end{equation}
We use the convention that the infimum over the empty set is $+\infty.$
\begin{theorem} \label{thm:good-rate-function}
  For any $x \in H$, $I_x$ is a good rate function
\end{theorem}
The proof is given in Appendix \ref{S:good}.
\vspace{.3cm}

The main theorem of this section is below.
\begin{theorem} \label{thm:example-fwuldp}
  If we assume \eqref{eq:G-bounded}, then let $\mathscr{A}$ be the collection of all subsets of $H$. If we assume \eqref{eq:G-grows}, then let $\mathscr{A}$ be the collection of all bounded subsets of $H$. For any $T>0$, $X^\e_x$ satisfies a FWULDP in $\mathcal{E} = C([0,T]:H)$ with rate function $I_x$ and speed $a(\e) = \e$ uniformly over $\mathscr{A}$.
\end{theorem}
This theorem demonstrates that compact sets are not required for the large deviations principle to hold. Bounded subsets of infinite dimensional Hilbert spaces are not generally compact. Furthermore, if $G$ is bounded in such a way that \eqref{eq:G-bounded} holds, then the large deviations principle is uniform over all sets of initial conditions, including uniformity over the whole space $H$.

Based on the equivalence of the EULP and the FWULDP (Theorem \ref{thm:FWULDP=EULP}) along with Theorem \ref{thm:eulp-sufficient},  Theorem \ref{thm:example-fwuldp} will be an immediate consequence of the following result.

\begin{theorem} \label{thm:control-conv}
  If we assume \eqref{eq:G-bounded}, then let $\mathscr{A}$ be the collection of all subsets of $H$. If we assume \eqref{eq:G-grows}, then let $\mathscr{A}$ be the collection of all bounded subsets of $H$. For any $T>0$, $N>0$, and $A \in \mathscr{A}$,
  \begin{equation}
    \lim_{\e \to 0} \sup_{x \in A}\sup_{u \in\mathcal{P}_2^N} \Pro\left( \left|X^{\e,u}_x - X^{0,u}_x \right|_{C([0,T]:H)} >\delta \right) = 0.
  \end{equation}
\end{theorem}

The proof of Theorem \ref{thm:control-conv} is based on the following lemmas whose proofs we sketch in Appendix \ref{S:lemmas}.
For any $\mathcal{F}_t$-adapted $\varphi \in C([0,T]:H)$ define the stochastic convolution by
\begin{equation}
  \Gamma(\varphi)(t) = \int_0^t S(t-s)G(\varphi(s))dw(s).
\end{equation}
For any $\varphi \in C([0,T]:H)$ and $u \in L^2([0,T]:U)$ define the controlled convolution
\begin{equation} \label{eq:Lambda}
  [\Lambda(\varphi)u](t) = \int_0^t S(t-s)G(\varphi(s))u(s)ds.
\end{equation}
For any $\varphi \in C([0,T]:H)$ define the nonlinear convolution
\begin{equation}
  \Theta(\varphi)(t) = \int_0^t S(t-s)B(\varphi(s))ds.
\end{equation}
In this notation the mild formulation for the stochastic controlled equation \eqref{eq:control-mild} can be written as
\begin{equation} \label{eq:mild-w-symbols}
  X^{\e,u}_x(t) = S(t)x + \Theta(X^{\e,u}_x)(t) + \sqrt{\e}\Gamma(X^{\e,u}_x)(t) + [\Lambda(X^{\e,u}_x)u](t).
\end{equation}

\begin{lemma} \label{lem:Gamma}
  For any $T>0$ and $p> \frac{1}{\alpha}$ where $\alpha$ is  from \eqref{eq:K-int}, there exists a constant $C=C(T,p)$ such that
  \begin{enumerate}
    \item For any $\mathcal{F}_t$-adapted $\varphi, \psi \in C([0,T]:H)$ and  $t \in [0,T]$
    \begin{equation} \label{eq:Gamma-Lip}
      \E|\Gamma(\varphi) - \Gamma(\psi)|_{C([0,t]:H)}^p \leq C \E\int_0^t |\varphi -\psi|_{C([0,s]:H)}^p ds.
    \end{equation}
    \item If \eqref{eq:G-bounded} holds, then for any $\mathcal{F}_t$-adapted $\varphi \in C([0,T]:H)$ and $t \in [0,T]$,
    \begin{equation} \label{eq:Gamma-bound}
      \E|\Gamma(\varphi)|_{C([0,t]:H)}^p \leq C.
    \end{equation}
    \item If \eqref{eq:G-grows} holds, then for any $\mathcal{F}_t$-adapted $\varphi \in C([0,T]:H)$ and $t \in [0,T]$,
    \begin{equation}\label{eq:Gamma-grow}
      \E|\Gamma(\varphi)|_{C([0,t]:H)}^p \leq C \left( 1 + \E\int_0^t |\varphi|_{C([0,s]:H)}^p ds \right).
    \end{equation}
  \end{enumerate}
\end{lemma}

\begin{lemma} \label{lem:Lambda}
  For any $T>0$ and $p> \frac{1}{\alpha}$ where $\alpha$ is  from \eqref{eq:K-int}, there exists $C = C(T,p)$ such that
  \begin{enumerate}
    \item For any $\varphi,\psi \in C([0,T]:H)$, $u \in L^2([0,T]:U)$, and $t \in [0,T]$,
    \begin{equation} \label{eq:Lambda-Lip}
      |\Lambda(\varphi)u - \Lambda(\psi)u|_{C([0,t]:H)}^p \leq C|u|_{L^2([0,t]:U)}^p \int_0^t |\varphi - \psi|_{C([0,s]:H)}^p ds.
    \end{equation}
    \item If \eqref{eq:G-bounded} holds, then for any $\varphi \in C([0,T]:H)$, $u \in L^2([0,T]:U)$, and $t \in [0,T]$,
    \begin{equation} \label{eq:Lambda-bound}
      |\Lambda(\varphi)u|_{C([0,t]:H)}^p \leq C |u|_{L^2([0,t]:U)}^p.
    \end{equation}
    \item If \eqref{eq:G-grows} holds, then for any $\varphi \in C([0,T]:H)$, $u \in L^2([0,T]:U)$, and $t \in [0,T]$,
    \begin{equation} \label{eq:Lambda-grow}
      |\Lambda(\varphi)u|_{C([0,t]:H)}^p \leq C |u|_{L^2([0,t]:U)}^p \left( 1 +  \int_0^t |\varphi|_{C([0,s]:H)}^p ds \right).
    \end{equation}
  \end{enumerate}
\end{lemma}

\begin{lemma} \label{lem:Theta}
  For any $T>0$ and $p> 1$, there exists $C = C(T,p)$ such that
  \begin{enumerate}
    \item For any $\varphi, \psi \in C([0,T]:H)$ and $t \in [0,T]$,
    \begin{equation} \label{eq:Theta-Lip}
      |\Theta(\varphi) - \Theta(\psi)|_{C([0,t]:H)}^p \leq C \int_0^t |\varphi - \psi|_{C([0,s]:H)}^p ds.
    \end{equation}
    \item For any $\varphi \in C([0,T]:H)$ and $t \in [0,T]$,
    \begin{equation} \label{eq:Theta-grow}
      |\Theta(\varphi)|_{C([0,t]:H)}^p \leq C \left( 1 + \int_0^t |\varphi|_{C([0,s]:H)}^p ds \right).
    \end{equation}
  \end{enumerate}
\end{lemma}

\begin{lemma} \label{lem:a-priori}
  Assuming either \eqref{eq:G-bounded} or \eqref{eq:G-grows}, for any $x \in H$, $\e>0$, $N>0$, and $u \in \mathcal{P}_2^N$, there exists a unique, continuous $\mathcal{F}_t$-adapted solution to \eqref{eq:control-mild}. Furthermore, for any $T>0$ and $p > \frac{1}{\alpha}$, there exists $C=C(T,p)$ such that for any $\e>0$, $N>0$, and $R>0$,  $X^{\e,u}_x$ satisfies the bound
  \begin{equation} \label{eq:a-priori}
    \sup_{|x|_H \leq R} \sup_{u \in \mathcal{P}_2^N}\E |X^{\e,u}_x|_{C([0,T]:H)}^p \leq C(1 +R^p + N^{\frac{p}{2}}+\e^{\frac{p}{2}}) e^{CT( 1 + N^{\frac{p}{2}}+\e^{\frac{p}{2}})}.
  \end{equation}
\end{lemma}
Lemma \ref{lem:a-priori} is a straightforward consequence of Lemmas \ref{lem:Gamma}, \ref{lem:Lambda}, and \ref{lem:Theta}. The existence and uniqueness proof is a standard argument based on Picard iteration. The $L^p$ bound proof is a straightforward application of Gr\"onwall's inequality.

\begin{proof}[Proof of Theorem \ref{thm:control-conv}]
  Fix $T>0$, $p>\frac{1}{\alpha}$, and $N>0$. In this proof, $C$ represents an arbitrary constant independent of $\e,x,u$ and $N$ whose value will change from line to line. By the notation of \eqref{eq:mild-w-symbols}, for any $x \in H$, $\e>0$, and $u \in \mathcal{P}_2^N$,
  \[X^{\e,u}_x - X^{0,u}_x= \Theta(X^{\e,u}_x) - \Theta(X^{0,u}_x) + \sqrt{\e} \Gamma(X^{\e,u}_x) + \Lambda(X^{\e,u}_x) - \Lambda(X^{0,u}_x). \]

  It follows from \eqref{eq:Lambda-Lip} and \eqref{eq:Theta-Lip} that for any $t \in [0,T]$,
  \begin{align*}
    |X^{\e,u}_x - X^{0,u}_x|_{C([0,t]:H)}^p \leq &C\e^{\frac{p}{2}}|\Gamma(X^{\e,u}_x)|_{C([0,t]:H)}^p \\
    &+ C(1+N^{p/2}) \int_0^t |X^{\e,u}_x - X^{0,u}_x|_{C([0,s]:H)}^p ds.
  \end{align*}
  By Gr\"onwall's inequality,
  \begin{equation} \label{eq:gronwall}
    |X^{\e,u}_x - X^{0,u}_x|_{C([0,T]:H)}^p \leq C\e^{\frac{p}{2}} e^{C(1+N^{p/2})T} |\Gamma(X^{\e,u}_x)|_{C([0,T]:H)}^p.
  \end{equation}

  If we assume \eqref{eq:G-bounded} so that $G$ is bounded, then \eqref{eq:Gamma-bound} holds. Consequently,
  \[\E|X^{\e,u}_x - X^{0,u}_x|_{C([0,T]:H)}^p \leq C \e^{\frac{p}{2}}e^{C(1+N^{p/2})T} .\]
  This bound is independent of $x \in H$. Therefore,
  \[\lim_{\e \to 0} \sup_{x \in H} \sup_{u \in \mathcal{P}_2^N}\E|X^{\e,u}_x - X^{0,u}_x|_{C([0,T]:H)}^p =0. \]
  The result follows by Chebyshev inequality.

  On the other hand, if we assume \eqref{eq:G-grows},  then  \eqref{eq:a-priori} and \eqref{eq:Gamma-grow} imply that if we restrict $x$ to bounded sets that $\E|\Gamma(X^{\e,u}_x)|_{C([0,T]:H)}^p$ will be bounded. In particular for $R>0$ and $N>0$ it follows from \eqref{eq:gronwall} that
  \[\lim_{ \e \to 0} \sup_{|x|_H \leq R} \sup_{u \in \mathcal{P}_2^N} \E |X^{\e,u}_x - X^{0,u}_x|_{C([0,T]:H)}^p =0.\]
  The result follows by Chebyshev inequality.
\end{proof}

\section{Equivalence of the FWULDP and the ULP -- Proof of Theorem \ref{thm:FWULDP=ULP}} \label{S:Proof-FWULDP=ULP}
  For this section, assume that Assumption \ref{assum:compact-level-union} holds. As we showed in Theorem \ref{thm:ulp-counter}, the lack of uniform continuity of the bounded continuous function $h$, leads to counterexamples where the FWULDP holds but the ULP does not.
  The compactness of $\bigcup_{x \in A} \Phi_x(s)$ in Assumption \ref{assum:compact-level-union} is important assumption because continuous functions are uniformly continuous over compact sets.
  \begin{lemma} \label{lem:unif-cont}
    Let $K \subset \mathcal{E}$ be compact and let $h: \mathcal{E} \to \mathbb{R}$ be a continuous function. Then $h$ is uniformly continuous near $K$ in the sense that for any $\eta>0$ there exists $\delta>0$ such that for all $\varphi \in K$ and $\psi \in \mathcal{E}$ such that $\rho(\psi,\varphi)<\delta$, it follows that $|h(\varphi) - h(\psi)|< \eta$.
  \end{lemma}
  We omit the proof because this result is classical.
  \vspace{.3cm}

  \begin{lemma} \label{lem:ulp->fw-low}
    Under Assumption \ref{assum:compact-level-union}, the ULP implies the FWULDP lower bound \eqref{eq:fwuldp-lower}.
  \end{lemma}
  \begin{proof}
     Assume that $X^\e_x$ satisfies a ULP with respect to rate function $I_x$ with speed $a(\e)$ uniformly over $\mathscr{A}$. Fix $A \in \mathscr{A}$, $s_0>0$, and $\delta>0$. Let $x_n \in A$, $\varphi_n \in \Phi_{x_n}(s_0)$ and  $ \e_n \downarrow 0$ be arbitrary sequences. By Assumption \ref{assum:compact-level-union}, $\bigcup_{x \in A}\Phi_x(s_0)$ is a precompact set. Therefore, there exists a subsequence (relabeled $(x_n, \varphi_n, \e_n)$) and a limit $\tilde\varphi \in \mathcal{E}$ such that $\varphi_n \to \tilde \varphi$ in $\mathcal{E}$. There exists $N \in \mathbb{N}$ such that for all $n \geq N$, $\rho(\varphi_n, \tilde \varphi)< \frac{\delta}{2}$. Then if $n\geq N$, $\{\varphi\in \mathcal{E}: \rho(\varphi, \tilde \varphi)< \frac{\delta}{2}\} \subset \{\varphi\in \mathcal{E}: \rho(\varphi, \varphi_n) < \delta\}$. Consequently, for $n \geq N$,
     \begin{equation} \label{eq:tilde-varphi-ineq}
         \Pro(\rho(X^{\e_n}_{x_n}, \tilde \varphi)<\delta/2) \leq \Pro(\rho(X^{\e_n}_{x_n}, \varphi_n)< \delta).
     \end{equation}
     Let $j>s_0$ and define the bounded continuous function $h: \mathcal{E} \to \mathbb{R}$ by
     \begin{equation} \label{eq:h-def-fw-low}
       h(\psi) = j \min \left\{ \frac{2 \rho(\psi, \tilde \varphi)}{\delta},1 \right\}.
     \end{equation}
     This function has the properties that $h \geq 0$, $h(\tilde{\varphi}) = 0$,  and   $h(\psi) = j$ if $\rho(\psi, \tilde\varphi)\geq\frac{\delta}{2}$. Combining this observation with \eqref{eq:tilde-varphi-ineq}, it follows that for any $n \geq N$,
     \begin{equation*}
       \E \exp \left( -\frac{h(X^{\e_n}_{x_n})}{a(\e_n)} \right) \leq e^{-\frac{j}{a(\e_n)}} + \Pro(\rho(X^{\e_n}_{x_n}, \varphi_n)< \delta)
     \end{equation*}
     and
     \begin{align} \label{eq:exp-upper-bound}
      &a(\e_n) \log \E \exp \left( -\frac{h(X^{\e_n}_{x_n})}{a(\e_n)} \right) \nonumber\\
      &\leq a(\e_n) \log (2) + \max\{ -j, a(\e) \log \Pro(\rho(X^{\e_n}_{x_n}, \varphi_n)< \delta)\}.
     \end{align}

     Next we observe that because $\varphi_n \in \mathcal{E} $,
     \begin{equation} \label{eq:inf-h-I-upper-bound}
       \inf_{\varphi \in \mathcal{E}} \{h(\varphi) + I_{x_n}(\varphi)\} \leq h(\varphi_n) + I_{x_n}(\varphi_n).
     \end{equation}

     Combining \eqref{eq:exp-upper-bound} and \eqref{eq:inf-h-I-upper-bound},
     \begin{align*}
       & \liminf_{n \to +\infty}\left( \max\{ -j, a(\e) \log \Pro(\rho(X^{\e_n}_{x_n}, \varphi_n)<\delta)\} + h(\varphi_n) + I_{x_n}(\varphi_n)\right)\\
       &\geq \liminf_{n \to +\infty}\left( a(\e_n) \log \E \left( -\frac{h(X^{\e_n}_{x_n})}{a(\e_n)} \right) + \inf_{\varphi \in \mathcal{E}} \{h(\varphi) + I_{x_n}(\varphi) \}\right)\\
       &\geq 0.
     \end{align*}
     The last inequality follows from \eqref{eq:ulp} because we assumed that $X^\e_x$ satisfies a ULP. By the continuity of $h$, $\lim_{n \to +\infty} h(\varphi_n) = h(\tilde\varphi) = 0$. We recall that $I_{x_n}(\varphi_n) \leq s_0$ and that $j>s_0$ implying that $-j + \liminf_{n \to +\infty} I_{x_n}(\varphi_n) < 0$.  From these observation and the above display we conclude that
     \[\liminf_{ n \to+\infty} \left(\Pro(\rho(X^{\e_n}_{x_n}, \varphi_n)< \delta) + I_{x_n}(\varphi_n) \right) \geq 0.\]
     Because the sequences $(x_n, \varphi_n, \e_n)$ were arbitrary, the FWULDP lower bound \eqref{eq:fwuldp-lower} follows.
  \end{proof}

  \begin{lemma} \label{lem:ulp->fw-up}
    Under Assumption \ref{assum:compact-level-union}, the ULP implies the FWULDP upper bound \eqref{eq:fwuldp-upper}.
  \end{lemma}

  \begin{proof}
    Assume that $X^\e_x$ satisfies a ULP with respect to rate function $I_x$ with speed $a(\e)$ uniformly over $\mathscr{A}$.
    Fix $A \in \mathscr{A}$, $\delta>0$, and $s_0>0$. Let $x_n \in A$, $s_n\in [0,s_0]$,  $\e_n\downarrow 0$ be arbitrary sequences. By Assumption \ref{assum:compact-level-union}, $\bigcup_{x \in A} \Phi_x(s_0)$ is a precompact set. This means that the collection of closed subsets of its closure $\overline{\bigcup_{ x \in A} \Phi_x(s_0)}$ form a compact metric space under the Hausdorff metric \eqref{eq:Hausdorff}. Because we assumed that each $I_x$ is a lower-semicontinuous rate function, it follows that for each $n \in \mathbb{N}$, $\Phi_{x_n}(s_n)$ is a closed subset of $\overline{\bigcup_{x \in A}\Phi_{x_n}(s_0)}$. By the compactness of the Hausdorff metric space, there exists a subsequence (relabeled $(x_n, s_n, \e_n)$) and a closed set $B \subset \overline{\bigcup_{x \in A} \Phi_x(s_0)}$ such that
    $
      \lim_{n \to +\infty} \lambda(\Phi_{x_n}(s_n), B) = 0.
    $
    There must exist $N_1 \in \mathbb{N}$ such that for all $n \geq N_1$, it follows that
    $
      \lambda(\Phi_{x_n}(s_n), B)< \frac{\delta}{2}.
    $

    A consequence of this is that when $n \geq N_1$,
    \[\{\varphi \in \mathcal{E}: \dist(\varphi, \Phi_{x_n}(s_n))\geq \delta\} \subset\{\varphi \in \mathcal{E}: \dist(\varphi, B)\geq \frac{\delta}{2}.\}\] Therefore,
    \begin{equation} \label{eq:B-ineq}
      \Pro(\dist(X^{\e_n}_{x_n}, \Phi_{x_n}(s_n)) \geq \delta) \leq \Pro(\dist(X^{\e_n}_{x_n}, B)\geq \delta/2).
    \end{equation}

    Now we define a bounded continuous function $h: \mathcal{E} \to \mathbb{R}$. Let $j>s_0$ and define
    \begin{equation} \label{eq:h-def-fw-up}
      h(\psi) = j - j \min\left\{\frac{2\dist(\psi,B)}{\delta}, 1  \right\}.
    \end{equation}
    This function has the properties that $h(\psi) = 0$ if $\dist(\psi, B) \geq \frac{\delta}{2}$ and $h(\psi) = j$ if $\psi \in B$. One consequence of these properties is that for  $n \geq N_1$
    \begin{equation} \label{eq:exp-lower-bound}
      \Pro(\dist(X^{\e_n}_{x_n}, B) \geq \delta/2) \leq \E \exp \left( - \frac{h(X^{\e_n}_{x_n})}{a(\e_n)} \right).
    \end{equation}
    Combining \eqref{eq:B-ineq} and \eqref{eq:exp-lower-bound}, for $n \geq N_1$,
    \begin{equation} \label{eq:exp-upper-lower-2}
      a(\e_n) \log \Pro(\dist(X^{\e_n}_{x_n}, \Phi_{x_n}(s_n)) \geq \delta) \leq a(\e_n) \log \E \exp\left( - \frac{h(X^{\e_n}_{x_n})}{a(\e_n)} \right).
    \end{equation}

    Because $j$ was chosen to be larger than $s_0$ and $\lambda(\Phi_{x_n}(s_n), B)$ converges to zero, there exists $N_2\geq N_1$ such that for all $n \geq N_2$, $$j- j \delta^{-1}2\lambda(\Phi_{x_n}(s_n), B)> s_0.$$
    If $n \geq N_2$, then for any $\varphi \in \Phi_{x_n}(s_n)$  the definition of the Hausdorff metric guarantees that $\dist(\varphi, B) \leq \lambda(\Phi_{x_n}(s_n), B)$ and that
    \[h(\varphi) = j - j\min\left\{ \frac{2\dist(\varphi, B)}{\delta},1 \right\} \geq j - j\delta^{-1} 2\lambda(\Phi_{x_n}(s_n), B) > s_0 \geq s_n.\]
    On the other hand, if $\varphi \not \in \Phi_{x_n}(s_n)$ then $h(\varphi) \geq 0$ (because it is always positive) and $I_{x_n}(\varphi_n) \geq s_n$.
    From these observations it follows that for $n \geq N_2$,
    \begin{equation} \label{eq:inf-h-I-lower-bound}
      \inf_{\varphi \in \mathcal{E}} \{h(\varphi) + I_{x_n}(\varphi)\} \geq s_n.
    \end{equation}

    Combining \eqref{eq:exp-upper-lower-2} and \eqref{eq:inf-h-I-lower-bound},
    \begin{align*}
      &\limsup_{n \to +\infty} \left(a(\e_n) \log \Pro(\dist(X^{\e_n}_{x_n}, \Phi_{x_n}(s_n))\geq \delta) + s_n \right)\\
      &\geq \limsup_{n \to +\infty} \left(a(\e_n) \log \E \exp \left(- \frac{h(X^{\e_n}_{x_n})}{a(\e_n)} \right) + \inf_{\varphi \in \mathcal{E}} \{h(\varphi) + I_{x_n}(\varphi)\} \right)\\
      &\geq 0.
    \end{align*}
    The last inequality follows because we assumed that $X^\e_x$ satisfies a ULP \eqref{eq:ulp}. We can conclude that the FWULDP upper bound \eqref{eq:fwuldp-upper} follows because the original sequences $(x_n, s_n, \e_n)$ were arbitrary.
  \end{proof}

  \begin{lemma} \label{lem:fw->ulp-low}
    Under Assumption \ref{assum:compact-level-union}, the FWULDP implies that for any $A \in \mathscr{A}$ and any bounded, continuous $h: \mathcal{E} \to \mathbb{R}$.
    \begin{equation} \label{eq:ulp-lower}
      \liminf_{\e \to 0} \inf_{x \in A} \left( a(\e) \log \E \exp \left(- \frac{h(X^\e_x)}{a(\e)} \right) + \inf_{\varphi \in \mathcal{E}}\{h(\varphi)+ I_x(\varphi)\} \right) \geq 0.
    \end{equation}
  \end{lemma}
  \begin{proof}
    Assume that $X^\e_x$ satisfies a FWULDP with respect to rate function $I_x$ with speed $a(\e)$ uniformly over $\mathscr{A}$.
    Fix $A \in \mathscr{A}$ and a bounded continuous $h: \mathcal{E} \to \mathbb{R}$. Let $\eta>0$ be arbitrary.  For each $x \in A$, there exists $\varphi_x \in \mathcal{E}$ such that
    \begin{equation} \label{eq:phi-x-choice}
       h(\varphi_x) + I_{x}(\varphi_x)  \leq \inf_{\varphi \in \mathcal{E}} \{h(\varphi) + I_{x_n}(\varphi)\} +  \frac{\eta}{2}.
    \end{equation}
    Let $s_0 = 2\|h\|_{C(\mathcal{E})} +\eta/2$. By Lemma \ref{lem:inf-bound} $\varphi_x \in \Phi_{x}(s_0)$ for all $x \in A$.

    By Assumption \ref{assum:compact-level-union} and Lemma \ref{lem:unif-cont}, there exists $\delta>0$ such that for all $\varphi \in \bigcup_{x \in A} \Phi_x(s_0)$ and $\psi \in \mathcal{E}$ such that $\rho(\varphi, \psi)< \delta$, it follows that $|h(\varphi) - h(\psi)|< \eta/2$. Consequently, for any $\e>0$ and $x \in A$,
    \begin{align} \label{eq:exp-lower-bound-fw}
      &\E \exp \left( -\frac{h(X^{\e}_{x})}{a(\e)} \right) \geq \E \left[\exp\left(-\frac{h(X^{\e}_{x})}{a(\e)} \right) \mathbbm{1}_{\{\rho(X^{\e}_{x},\varphi_x)<\delta\}} \right] \nonumber\\
      &\geq \exp\left(-\frac{(h(\varphi_x) + \eta/2)}{a(\e)}   \right) \Pro(\rho(X^{\e}_{x},\varphi_x)<\delta).
    \end{align}

    By the FWULDP lower bound \eqref{eq:fwuldp-lower} and by \eqref{eq:phi-x-choice} and \eqref{eq:exp-lower-bound-fw},
    \begin{align*}
      &\liminf_{\e \to 0} \inf_{x \in A} \left( a(\e) \log \E \exp \left( -\frac{h(X^{\e}_{x})}{a(\e)} \right) + \inf_{\varphi \in \mathcal{E}}\{h(\varphi) + I_{x}(\varphi)\} \right)\\
      &\geq \liminf_{\e \to 0}\inf_{x \in A} \Big(-h(\varphi_x) - \eta/2 + a(\e) \log \Pro(\rho(X^{\e}_{x},\varphi_x)< \delta)\\
      &\hspace{3cm} + h(\varphi_x) +I_{x}(\varphi_x) - \eta/2\Big)\\
      &\geq \liminf_{\e\to 0} \inf_{x \in A} \inf_{\varphi \in \Phi_x(s_0)} \left(a(\e) \log \Pro(\rho(X^\e_x, \varphi)<\delta) + I_x(\varphi) \right)-\eta\\
      & \geq -\eta.
    \end{align*}
    Because $\eta>0$ was arbitrary, \eqref{eq:ulp-lower} follows.
  \end{proof}

  \begin{lemma} \label{lem:fw->ulp-up}
    Under Assumption \ref{assum:compact-level-union}, the FWULDP implies that for any $A \in \mathscr{A}$ and any bounded, continuous $h: \mathcal{E} \to \mathbb{R}$.
    \begin{equation} \label{eq:ulp-upper}
      \limsup_{\e \to 0} \sup_{x \in A} \left( a(\e) \log \E \exp \left(- \frac{h(X^\e_x)}{a(\e)} \right) + \inf_{\varphi \in \mathcal{E}}\{h(\varphi)+ I_x(\varphi)\} \right) \leq 0.
    \end{equation}
  \end{lemma}

  \begin{proof}
    Assume that $X^\e_x$ satisfies a FWULDP with respect to rate function $I_x$ with speed $a(\e)$ uniformly over $\mathscr{A}$.
    Fix a bounded continuous $h: \mathcal{E} \to \mathbb{R}$ and $\eta>0$. Let $s_0 = 2\|h\|_{C(\mathcal{E})}+\eta$. By Assumption \ref{assum:compact-level-union} and Lemma \ref{lem:unif-cont}, there exists $\delta>0$ such that for all $\varphi \in \bigcup_{x \in A} \Phi_x(s_0)$ and $\psi \in \mathcal{E}$ such that $\rho(\psi,\varphi)<\delta$, it follows that $|h(\varphi)-h(\psi)|<\eta/2$.

    Let $N \in \mathbb{N}$ be such that $ N\eta/2 > s_0$. For $x \in A$, define the subsets of $\mathcal{E}$,
    \begin{align*}
      &E_0^x = \left\{\varphi \in \mathcal{E}: \dist(\varphi, \Phi_x(\eta/2))<\delta \right\}\\
      &E^x_k = \left\{\varphi \in \mathcal{E}: \dist(\varphi, \Phi_x(k\eta/2))\geq \delta, \ \dist(\varphi, \Phi_x((k+1)\eta/2))<\delta \right\}, \\ & \text{ for } k=1,...,N-1\\
      &E^x_N = \left\{\varphi \in \mathcal{E}: \dist(\varphi, \Phi_x(N\eta/2))\geq \delta \right\}.
    \end{align*}
    Note that $\bigcup_{k=0}^N E^x_k = \mathcal{E}$.
    For any $x \in A$ and $\e>0$,
    \begin{align*}
      &\E \exp \left( -\frac{h(X^\e_x)}{a(\e)}\right) \leq \sum_{k=0}^N \E \left(\exp\left( - \frac{h(X^\e_x)}{a(\e)} \right) \mathbbm{1}_{\{X^\e_x \in E_k^x\}} \right)\\
      &\leq \sum_{k=0}^N \exp\left(-\frac{\inf_{\varphi \in E_k^x} h(\varphi)}{a(\e)} \right) \Pro(X^\e_x \in E_k^x).
    \end{align*}
    It follows that
    \begin{align*}
      & a(\e) \log \E \exp \left( -\frac{h(X^\e_x)}{a(\e)}\right)\nonumber\\
      &\leq a(\e) \log \left(\sum_{k=0}^N \exp\left(-\frac{\inf_{\varphi \in E_k^x} h(\varphi)}{a(\e)} \right) \Pro(X^\e_x \in E_k^x) \right)\nonumber\\
      & \leq a(\e) \log(N+1) + \max_{k \in \{0,...,N\}} \left\{-\inf_{\varphi \in E_k^x} h(\varphi) + a(\e) \log \Pro(X^\e_x \in E_k^x) \right\}.
    \end{align*}
    By adding and subtracting $\frac{k\eta}{2}$,
    \begin{align}\label{eq:log-estimate}
      & a(\e) \log \E \exp \left( -\frac{h(X^\e_x)}{a(\e)}\right)\nonumber\\
      & \leq a(\e) \log(N+1) + \max_{k \in \{0,...,N\}} \left\{  a(\e) \log \Pro(X^\e_x \in E_k^x) + k\eta/2 \right\}\nonumber\\
      &\qquad+ \max_{k \in \{0,...,N\}} \left\{-\inf_{\varphi \in E_k^x} h(\varphi) - k\eta/2 \right\}\nonumber\\
      &\leq a(\e) \log(N+1) + \max_{k \in \{0,...,N\}} \left\{  a(\e) \log \Pro(X^\e_x \in E_k^x) + k\eta/2 \right\}\nonumber\\
       &\qquad- \min_{k \in \{0,...,N\}} \left\{\inf_{\varphi \in E_k^x} h(\varphi) + k\eta/2 \right\}.
    \end{align}
    By the definition of $E_k^x$ for $k\in\{1,..,N\}$,
    $$\Pro(X^\e_x \in E_k^x) \leq \Pro(\dist(X^\e_x,\Phi_x(k\eta/2))\geq\delta)$$
     and it follows by the FWULDP upper bound \eqref{eq:fwuldp-upper} that
    \begin{align*}
      &\limsup_{\e \to 0} \sup_{x \in A} \max_{k \in \{1,...,N\}} \{a(\e) \log \Pro(X^\e_x \in E_k^x) + k\eta/2\} \\
      &\leq \limsup_{\e \to 0} \sup_{x \in A} \max_{k \in \{1,...,N\}} \{a(\e) \log \Pro(\dist(X^\e_x,\Phi_x(k\eta/2))\geq \delta) + k\eta/2\} \\
      &\leq 0.
    \end{align*}
    The $k=0$ case is trivially true because $ \Pro(X^\e_x \in E^x_0) \leq 1$ so it follows that
    \begin{equation} \label{eq:Ek-estimate}
      \lim_{\e \to 0} \sup_{x \in A} \max_{k \in \{0,...,N\}} \{a(\e) \log \Pro(X^\e_x \in E_k^x) + k\eta/2\} \leq 0.
    \end{equation}
    To prove \eqref{eq:ulp-upper} we show that for any $x \in A$,
    \begin{equation} \label{eq:inf-leq-inf}
      \inf_{\varphi \in \mathcal{E}} \{h(\varphi) + I_x(\varphi)\} \leq \min_{k \in \{0,..,N\}} \left\{\inf_{\varphi \in E_k^x} h(\varphi) + k\eta/2\right\} + \eta.
    \end{equation}
    Fix $k \in \{0,...,N-1\}$. Let $\varphi \in E_k^x$ be arbitrary. By the definition of $E_k^x$, there exists $\tilde\varphi \in \Phi_x((k+1)\eta/2)$ such that $\rho(\tilde{\varphi}, \varphi)< \delta$.  At the beginning of the proof we chose $\delta$ in a way that guarantees that $h(\tilde{\varphi})  \leq h(\varphi) + \eta/2$. Note that $I_x(\tilde\varphi) \leq k\eta/2 + \eta/2$. Therefore,
    \begin{equation*}
      (h(\varphi) + k\eta/2) + \eta \geq h(\tilde \varphi) + I_x(\tilde\varphi) \geq \inf_{\phi \in \mathcal{E}} \{h(\phi) + I_x(\phi)\}.
    \end{equation*}
    For fixed $x \in A$ and $k \in \{0,...,N-1\}$, take the infimum over $\varphi \in E_k^x$,
    \begin{equation*}
      \inf_{\varphi \in E_k^x} h(\varphi) + k\eta/2 + \eta \geq \inf_{\varphi \in \mathcal{E}} \{h(\varphi) + I_x(\varphi)\}.
    \end{equation*}
    The above inequality also holds for $k=N$ because of Lemma \ref{lem:inf-bound} and our choice of $N$. Therefore, \eqref{eq:inf-leq-inf} holds.

    By \eqref{eq:log-estimate}, \eqref{eq:Ek-estimate}, and \eqref{eq:inf-leq-inf},
    \begin{align*}
      &\limsup_{\e \to 0} \sup_{x \in A} \left(a(\e)\log \E \exp \left(- \frac{h(X^\e_x)}{a(\e)} \right) + \inf_{\varphi \in \mathcal{E}} \{h(\varphi) + I_x(\varphi)\} \right)\\
      &\leq \limsup_{\e \to 0} \left( a(\e)\log(N+1) +  \sup_{x \in A} \max_{k \in \{0,...,N\}} \{a(\e) \log \Pro(X^\e_x \in E_k^x) + k\eta/2\}
      \right)\\
      &\qquad + \sup_{x \in A} \left(\inf_{\varphi\in \mathcal{E}} \{h(\varphi) + I_x(\varphi)\} -  \min_{k \in \{0,...,N\}}\{\inf_{\varphi \in E_k^x} h(\varphi) + k\eta/2\} \right)\\
      & \leq \eta.
    \end{align*}
    Then \eqref{eq:ulp-upper} follows because $\eta>0$ was arbitrary.
  \end{proof}

  Theorem \ref{thm:FWULDP=ULP} follows from Lemmas \ref{lem:ulp->fw-low}, \ref{lem:ulp->fw-up}, \ref{lem:fw->ulp-low}, and \ref{lem:fw->ulp-up}.
\section{Equivalence of the FWULDP and the DZULDP -- Proof of Theorem \ref{thm:FWULDP=DZULDP}} \label{S:Proof-FWULDP=DZULDP}
In this section, we assume that Assumption \ref{assum:compact-initial-conds} holds.

\begin{lemma} \label{lem:fw->dz-low}
  Under Assumption \ref{assum:compact-initial-conds}, the FWULDP implies the DZULDP lower bound \eqref{eq:dzuldp-lower}.
\end{lemma}
\begin{proof}
  Assume that $X^\e_x$ satisfies a FWULDP with respect to rate function $I_x$ with speed $a(\e)$ uniformly over $\mathscr{A}$ where $\mathscr{A}$ satisfies Assumption \ref{assum:compact-initial-conds}.
  Let $A \in \mathscr{A}$ and let $G \subset \mathcal{E}$ be open. If $\sup_{x \in A} I_x(G) = +\infty$, then \eqref{eq:dzuldp-lower} is trivially true. Assume that $s_0:= \sup_{x \in A} I_x(G) <+\infty$.

  Let $x_n \in A$ and $\e_n\downarrow 0$ be arbitrary sequences. Let $\eta>0$. Because Assumption \ref{assum:compact-initial-conds} says that $A \subset \mathcal{E}_0$ is a compact set, there exists a subsequence (relabeled $(x_n,\e_n)$) and a limit $\tilde x \in A$ such that $x_n \to \tilde x$ in $\mathcal{E}_0$.

  Because of the definition of $s_0$, there must exist $\tilde\varphi \in G$ such that $I_{\tilde x}(\tilde \varphi) \leq I_{\tilde x}(G) + \eta \leq s_0 + \eta.$ Because $G$ is open, there exists $\delta>0$ such that $\{\varphi \in \mathcal{E}: \rho(\varphi,\tilde \varphi)<\delta\} \subset G$.

  By Assumption \ref{assum:compact-initial-conds}, the sets $\Phi_{x_n}(s_0 + \eta)$ converge to $\Phi_{\tilde x}(s_0 + \eta)$ in Hausdorff metric. In particular, there must exist a sequence $\{\varphi_n\} \subset \mathcal{E}$ such that $\varphi_n \in \Phi_{x_n}(s_0 + \eta)$ and $\varphi_n \to \tilde \varphi$. There must exist an $N>0$ such that for $n\geq N$, $\rho(\varphi_n,\varphi)<\delta/2$. In particular, for $n \geq N$,
  \[\{\varphi \in \mathcal{E}:\rho(\varphi,\varphi_n)<\delta/2\} \subset \{\varphi \in \mathcal{E}: \rho(\varphi,\tilde\varphi)<\delta\} \subset G.\]
   Therefore,
  \[\Pro(X^{\e_n}_{x_n} \in G) \geq \Pro(\rho(X^{\e_n}_{x_n} , \varphi_n)< \delta/2).\]
  By the FWULDP lower bound \eqref{eq:fwuldp-lower},
  \begin{align*}
    &\liminf_{n \to +\infty} \left(a(\e_n) \log \Pro(X^{\e_n}_{x_n} \in G) + I_{x_n}(\varphi_n) \right) \\
    &\geq \liminf_{n \to +\infty} \left(a(\e_n) \log \Pro(\rho(X^\e_{x_n}, \varphi_n)<\delta/2) + I_{x_n}(\varphi_n) \right)\geq 0.
  \end{align*}
  The $\varphi_n$ were chosen so that $I_{x_n}(\varphi_n) \leq s_0 + \eta$, so we can conclude that
  \[\liminf_{n \to +\infty} a(\e_n) \log \Pro(X^{\e_n}_{x_n} \in G) \geq -s_0 -\eta.\]
  \eqref{eq:dzuldp-lower} follows because the sequence $(x_n,\e_n)$ and $\eta>0$ were arbitrary.
\end{proof}

\begin{lemma} \label{lem:fw->dz-up}
  Under Assumption \ref{assum:compact-initial-conds}, the FWULDP implies the DZULDP upper bound \eqref{eq:dzuldp-upper}.
\end{lemma}

\begin{proof}
  Assume that $X^\e_x$ satisfies a FWULDP with respect to rate function $I_x$ with speed $a(\e)$ uniformly over $\mathscr{A}$ where $\mathscr{A}$ satisfies Assumption \ref{assum:compact-initial-conds}.
  Let $F \subset \mathcal{E}$ be closed and $A \in \mathscr{A}$. If $\inf_{x \in A} I_x(F) = 0$, then the lemma is trivially true. Assume $\inf_{x \in A} I_x(F) >0$ and let $0<s< \inf_{x \in A} I_x(F)$. Let $x_n \in A$ and $\e_n \downarrow 0$ be arbitrary.

  Because $A$ is compact by Assumption \ref{assum:compact-initial-conds}, there exists a subsequence (relabeled $(x_n, \e_n)$) and a limit $\tilde x \in A$ such that $x_n \to \tilde x$ in $\mathcal{E}_0$. Because $F$ is closed, $\Phi_{\tilde x}(s)$ is compact, and $F \cap \Phi_{\tilde x}(s)=\emptyset$, there must be some positive distance $\delta>0$ such that $F \subset \{\varphi \in \mathcal{E}: \dist(\varphi, \Phi_{\tilde x}(s))\geq \delta\}$.

  By Assumption \ref{assum:compact-initial-conds}, there exists $N \in \mathbb{N}$ such that for $n \geq N$, $$\lambda(\Phi_{x_n}(s), \Phi_{\tilde x}(s)) < \delta/2.$$ Therefore, for all $n \geq N$,
  $$F \subset \{\varphi \in \mathcal{E}: \dist(\varphi, \Phi_{\tilde x}(s))\geq \delta\} \subset \{\varphi \in \mathcal{E}: \dist(\varphi, \Phi_{x_n}(s))\geq \delta/2\}.$$
  By the FWULDP upper bound \eqref{eq:fwuldp-upper},
  \begin{align*}
  &\limsup_{n \to +\infty} a(\e_n) \log \Pro(X^{\e_n}_{x_n} \in F) \leq \limsup_{n \to +\infty} a(\e_n) \log \Pro(\dist(X^{\e_n}_{x_n},\Phi_{x_n}(s))\geq \delta/2) \\
  &\leq -s .
  \end{align*}
  The result follows because  the sequence $(x_n, \e_n)$ and $s< \inf_{x \in A} I_x(F)$ were arbitrary.
\end{proof}

\begin{lemma} \label{lem:dz->fw-low}
  Under Assumption \ref{assum:compact-initial-conds}, the DZULDP implies the FWULDP lower bound \eqref{eq:fwuldp-lower}.
\end{lemma}

\begin{proof}
  Assume that $X^\e_x$ satisfies a DZULDP with respect to rate function $I_x$ with speed $a(\e)$ uniformly over $\mathscr{A}$ where $\mathscr{A}$ satisfies Assumption \ref{assum:compact-initial-conds}.
  Fix $A \in \mathscr{A}$, $\delta>0$, and $s_0>0$. Let $x_n \in A$, $\varphi_n \in \Phi_{x_n}(s_0)$, $\e_n \downarrow 0$, and $\eta>0$ be arbitrary. By the compactness of $A$ and $[0,s_0]$, there exists a subsequence (relabeled $(x_n, \varphi_n, \e_n)$) and a limits $ \tilde x \in A$ and $s \in [0,s_0]$ such that $x_n \to \tilde x$  and $I_{x_n}(\varphi_n) \to s$. We choose this subsequence in such a way that $I_{x_n}(\varphi_n) \leq s + \eta$ for all $n$.

  By Assumption \ref{assum:compact-initial-conds}, $\lambda(\Phi_{x_n}(s + \eta), \Phi_{\tilde x}(s + \eta)) \to 0$. In particular, there must exist a sequence $\tilde \varphi_n \in \Phi_{\tilde x}(s + \eta)$ such that $\rho(\tilde \varphi_n, \varphi_n) \to 0$. By the compactness of $\Phi_{\tilde x}(s+\eta)$, there is a subsequence (relabeled $(x_n, \varphi_n, \e_n,\tilde \varphi_n)$) and a limit $\tilde \varphi \in \Phi_{\tilde x} (s + \eta)$ such that $\tilde{\varphi}_n \to \tilde \varphi$. It follows that $\varphi_n \to \tilde \varphi$ also.

  Define the open set $G = \{\varphi \in \mathcal{E} : \rho(\varphi, \tilde \varphi)< \delta/2\}$. Because $\varphi_n \to \tilde\varphi$, there exists $N \geq 0$ such that for $n\geq N$, $\rho(\varphi_n, \tilde \varphi)<\delta/2$. Therefore, $G \subset \{\varphi \in \mathcal{E}: \rho(\varphi, \varphi_n) < \delta\}$ and
  \begin{equation} \label{eq:fwdzbound}
    \Pro(\rho(X^{\e_n}_{x_n}, \varphi_n)<\delta) \geq \Pro(X^{\e_n}_{x_n} \in G).
  \end{equation}
  Also note that $\tilde \varphi \in G$ and for each $n \geq N$, $\varphi_n \in G$. Therefore,
  \begin{equation} \label{eq:IG-bound}
    I_{\tilde x} (\tilde \varphi) \geq I_{\tilde x}(G) \text{ and } I_{x_n}(\varphi_n) \geq I_{x_n}(G).
  \end{equation}

  Next, because of \eqref{eq:fwdzbound} and the fact that $I_{x_n}(\varphi_n) \to s$,
  \begin{align*}
    &\liminf_{n \to +\infty} \left( a(\e_n) \log \Pro(\rho(X^{\e_n}_{x_n}, \varphi_n)<\delta) + I_{x_n}(\varphi_n) \right)\\
    &\geq \liminf_{n \to +\infty} a(\e_n) \log \Pro(X^{\e_n}_{x_n} \in G) + s.
  \end{align*}
   Let $A_N  =\bigcup_{n=N}^\infty\{x_n\}\cup \{\tilde x\}$. $A_N$ is a compact subset of $\mathcal{E}_0$.
  Therefore, by the DZULDP lower bound\eqref{eq:dzuldp-lower},
  \[\liminf_{n \to +\infty} \left(a(\e_n) \log \Pro(X^{\e_n}_{x_n}, \varphi_n) + I_{x_n}(\varphi_n) \right)
    \geq -\sup_{y \in A_N} I_{y}(G) + s .\]
  By \eqref{eq:IG-bound} and the fact that $I_{x_n}(\varphi_n) \leq s + \eta$ and $I_{\tilde x}(\tilde \varphi) \leq s + \eta$, it follows that $\sup_{y \in A_N} I_y(G) \leq s+ \eta$ and therefore,
  \[\liminf_{n \to +\infty} \left(a(\e_n) \log \Pro(X^{\e_n}_{x_n}, \varphi_n) + I_{x_n}(\varphi_n) \right) \geq -\eta.\]
  The FWULDP lower bound \eqref{eq:fwuldp-lower} follows because the sequences $(x_n,\e_n,\varphi_n)$ and $\eta>0$ were arbitrary.
\end{proof}

\begin{lemma} \label{lem:dz->fw-up}
  Under Assumption \ref{assum:compact-initial-conds}, the DZULDP implies the FWULDP upper bound \eqref{eq:fwuldp-upper}.
\end{lemma}

\begin{proof}
  Assume that $X^\e_x$ satisfies a DZULDP with respect to rate function $I_x$ with speed $a(\e)$ uniformly over $\mathscr{A}$ where $\mathscr{A}$ satisfies Assumption \ref{assum:compact-initial-conds}.
  Fix $A \in \mathscr{A}$, $\delta>0$, and $s_0>0$. Fix $\eta>0$. Let $x_n\in  A$, $s_n \in [0,s_0]$ and $\e_n \downarrow 0$ be arbitrary.

  By the compactness of $A$ and $[0,s_0]$, there exist subsequences (relabeled $(x_n, s_n, \e_n)$) such that $x_n \to \tilde x \in A$ and $s_n \to s \in [0,s_0]$. We choose this subsequence in such a way that  for all $n$, it holds that $s_n> s-\eta$.

  Define the closed set $F = \{\varphi \in \mathcal{E}: \dist(\varphi, \Phi_{\tilde x}(s-\eta)) \geq \delta/2\}$.
  By Assumption \ref{assum:compact-initial-conds}, there exists $N \in \mathbb{N}$ such that for $n\geq N$, $\lambda(\Phi_{\tilde x}(s-\eta), \Phi_{x_n}(s-\eta))< \delta/4$.  Therefore, recalling that $s_n>s-\eta$, for $n\geq N$
  \begin{equation} \label{eq:F-subset}
    \{\varphi \in \mathcal{E}: \dist(\varphi, \Phi_{x_n}(s_n)) \geq \delta\} \subset \{\varphi \in \mathcal{E}: \dist(\varphi, \Phi_{x_n}(s-\eta))\geq \delta\} \subset F.
  \end{equation}
  Similarly for $n \geq N$,
  \begin{equation} \label{eq:F-supset}
    F \subset \{\varphi \in \mathcal{E}:\dist(\varphi, \Phi_{x_n}(s-\eta))\geq \delta/4\}
  \end{equation}
  Define the $\mathcal{E}_0$-compact set $A_N = \bigcup_{n=N}^\infty\{x_n\} \cup \{\tilde x\}$.
  It follows from the DZULDP upper bound \eqref{eq:dzuldp-upper} and \eqref{eq:F-subset} that
  \begin{align*}
    &\limsup_{n \to +\infty} \left( \e_n \log \Pro(\dist(X^{\e_n}_{x_n}, \Phi_{x_n}(s_n)) \geq \delta) + s_n \right)\\
  &\leq \limsup_{n \to +\infty} \e_n \log \Pro(X^{\e_n}_{x_n} \in F) + s\\
  & \leq -\inf_{y \in A_N} I_{y}(F) + s.
  \end{align*}
  By \eqref{eq:F-supset}, it follows that $F\cap \Phi_{x_n}(s-\eta)=\emptyset$ and $I_{x_n}(F) >s-\eta$. Similarly, $I_{\tilde x}(F) >s- \eta.$
  Therefore,
  \[\limsup_{n \to +\infty} \left( \e_n \log \Pro(\dist(X^{\e_n}_{x_n}, \Phi_{x_n}(s_n)) \geq \delta) + s_n \right)\leq \eta.\]
  Because  the sequences $(x_n,s_n,\e_n)$ and $\eta>0$ were arbitrary, the FWULDP upper bound \eqref{eq:fwuldp-upper} follows.
\end{proof}
Theorem \ref{thm:FWULDP=DZULDP} is a consequence of Lemmas \ref{lem:fw->dz-low}, \ref{lem:fw->dz-up}, \ref{lem:dz->fw-low}, and \ref{lem:dz->fw-up}.

\section{Equivalence of the FWULDP and EULP -- Proof of Theorem \ref{thm:FWULDP=EULP}} \label{S:Proof-FWULDP=EULP}

 \begin{lemma} \label{lem:eulp->fw-low}
   With no extra assumptions, the EULP implies the\\ FWULDP lower bound \eqref{eq:fwuldp-lower}.
 \end{lemma}
 \begin{proof}
   Assume that $X^\e_x$ satisfies an EULP with respect to rate function $I_x$ with speed $a(\e)$ uniformly over $\mathscr{A}$.
   Fix $A \in \mathscr{A}$, $\delta>0$, and $s_0>0$. Fix $j > s_0$. For any $\varphi \in \mathcal{E}$, define the test functions $h_{j,\delta,\varphi}(\psi) = j \min\left\{\frac{\rho(\psi,\varphi)}{\delta} , 1 \right\}$.
   These functions are uniformly bounded by $j$ and they are equicontinuous (actually equi-Lipschitz-continuous with Lipschitz constant $\frac{j}{\delta}$). With  $j$ and $\delta$ fixed, define the equibounded equicontinuous family of test functions $L:= \{h_{j,\delta,\varphi}: \varphi \in \bigcup_{x \in A} \Phi_x(s_0)\}$.

   Note that these functions have the properties that $h_{j,\delta,\varphi} \geq 0$ and \\$h_{j,\delta,\varphi}(\psi) = j$ if $\rho(\psi,\varphi)>\delta$. Therefore,
   $$\E \exp \left(- \frac{h_{j,\delta,\varphi}(X^\e_x)}{a(\e)} \right)\leq e^ {-\frac{j}{a(\e)} }+ \Pro(\rho(X^\e_x, \varphi)<\delta)$$ and
   \begin{align} \label{eq:exp-upper-bound-eulp}
     &a(\e) \log \E \exp \left(- \frac{h_{j,\delta,\varphi}(X^\e_x)}{a(\e)} \right) \nonumber \\
      &\leq a(\e) \log(2) + \max\left\{-j, a(\e) \log\Pro(\rho(X^\e_x, \varphi)<\delta)\right\}.
   \end{align}
   Furthermore, because $h_{j,\delta,\varphi}(\varphi) = 0$,
   \begin{equation} \label{eq:h-def-fw-up-eulp}
      \inf_{\phi \in \mathcal{E}}\{h_{j,\delta,\varphi}(\phi)  + I_x(\phi)\} \leq I_x(\varphi).
   \end{equation}

   Therefore, by \eqref{eq:exp-upper-bound-eulp}, \eqref{eq:h-def-fw-up-eulp}, and the EULP \eqref{eq:eulp},
   \begin{align*}
     &\liminf_{\e \to 0} \inf_{x \in A} \inf_{\varphi \in \Phi_x(s_0)} \left(\max\left\{-j, a(\e) \log\Pro(\rho(X^\e_x, \varphi)<\delta)   \right\}+ I_x(\varphi) \right)\\
     &\geq \liminf_{\e \to 0} \inf_{x \in A} \inf_{\varphi \in \Phi_x(s_0)} \Bigg( a(\e) \log \E \exp\left( -\frac{h_{j,\delta,\varphi}(X^\e_x)}{a(\e)}\right) \\
     &\hspace{4.5cm}+ \inf_{\phi \in \mathcal{E}}\{h_{j,\delta,\varphi}(\phi) + I_x(\phi)\}  \Bigg)\\
     &\geq \liminf_{\e \to 0} \inf_{x \in A} \inf_{h \in L} \left(a(\e) \log \E \exp\left(- \frac{h(X^\e_x)}{a(\e)} \right) + \inf_{\phi \in \mathcal{E}} \{h(\phi) + I_x(\phi)\}  \right)\\
     &\geq 0.
   \end{align*}
   Because we chose $j>s_0$, whenever $\varphi \in \Phi_x(s_0)$ it follows that $-j + I_x(\varphi)\leq -j+s_0<0$. We can conclude that
   \[\liminf_{\e \to 0} \inf_{x \in A} \inf_{\varphi \in \Phi_x(s_0)} \left( a(\e) \log \Pro(\rho(X^{\e}_{x}, \varphi_n)<\delta) + I_x(\varphi)   \right) \geq 0\]
   proving \eqref{eq:fwuldp-lower}.
 \end{proof}

 \begin{lemma} \label{lem:eulp->fw-up}
   With no extra assumptions, the EULP implies the\\ FWULDP upper bound \eqref{eq:fwuldp-upper}.
 \end{lemma}

 \begin{proof}
   Assume that $X^\e_x$ satisfies an EULP with respect to rate function $I_x$ with speed $a(\e)$ uniformly over $\mathscr{A}$.
   Fix $A \in \mathscr{A}$, $\delta>0$, and $s_0>0$. Let $j>s_0$. For any $x \in A$ and $s>0$, define the functions from $\mathcal{E} \to \mathbb{R}$
   \[h_{j,\delta,s,x}(\psi) = j - j \min\left\{\frac{\dist(\psi, \Phi_x(s))}{\delta} , 1 \right\}.\]
   For fixed $j$ and $\delta$, $L:= \{h_{j,\delta,s,x}: s \in [0,s_0], x \in A\}$ is a bounded equicontinuous family of functions, bounded by $j$ and with Lipschitz constant $\frac{j}{\delta}$. Observe that if $\dist(\psi, \Phi_x(s)) \geq \delta$ then $h_{j,\delta,s,x}(\psi) = 0$ implying that
   \begin{equation} \label{eq:exp-lower-bound-eulp}
     \E \exp \left(- \frac{h_{j,\delta,s,x}(X^\e_x)}{a(\e)} \right) \geq \Pro(\dist(X^\e_x, \Phi_x(s))\geq \delta).
   \end{equation}
   Note that either $\varphi \in \Phi_x(s)$, in which case $h_{j,\delta,s,x}(\varphi) = j>s$, or $\varphi \not \in \Phi_x(s)$, in which case $I_x(\varphi)>s$ implying that
   \begin{equation} \label{eq:inf-h-I-lower-bound-eulp}
     \inf_{\varphi \in \mathcal{E}} \{h_{j,\delta,s,x}(\varphi) + I_x(\varphi)\} \geq s.
   \end{equation}

   It follows from \eqref{eq:exp-lower-bound-eulp}, \eqref{eq:inf-h-I-lower-bound-eulp}, and the EULP \eqref{eq:eulp} that
   \begin{align*}
     &\limsup_{\e \to 0} \sup_{x \in A} \sup_{s \in [0,s_0]} \left(a(\e) \log \Pro(\dist(X^\e_x, \Phi_x(s))\geq \delta)  + s \right) \\
     &\leq  \limsup_{\e \to 0} \sup_{x \in A} \sup_{s \in[0, s_0]} \left(\E \exp \left(- \frac{h_{j,\delta,s,x}(X^\e_x)}{a(\e)} \right) + \inf_{\varphi \in \mathcal{E}} \{h_{j,\delta,s,x}(\varphi) + I_x(\varphi)\} \right)\\
     &\leq 0,
   \end{align*}
   proving \eqref{eq:fwuldp-upper}.
 \end{proof}

 \begin{lemma} \label{lem:fw->eulp-low}
   With no extra assumptions, the FWULDP implies the \\EULP lower bound. For any $A \in \mathscr{A}$ and family $L \subset C(\mathcal{E})$ of equibounded, equicontinuous functions from $\mathcal{E} \to \mathbb{R}$,
   \begin{equation} \label{eq:eulp-lower}
     \liminf_{\e \to 0} \inf_{x \in A} \inf_{h \in L} \left(a(\e) \log \E \exp \left( - \frac{h(X^\e_x)}{a(\e)}  \right) + \inf_{\varphi \in \mathcal{E}} \{h(\varphi) + I_x(\varphi)\} \right) \geq 0.
   \end{equation}
 \end{lemma}

 \begin{proof}
   Assume that $X^\e_x$ satisfies a FWULDP with respect to rate function $I_x$ with speed $a(\e)$ uniformly over $\mathscr{A}$.
   Let $A \in \mathscr{A}$ and $L$ be a family of uniformly bounded equicontinuous functions from $\mathcal{E} \to \mathbb{R}$. Fix $\eta>0$. For each $x \in A$ and $h \in L$, there exists $\varphi_{x,h} \in \mathcal{E}$ such that
   \begin{equation} \label{eq:h-upper-eulp}
     h(\varphi_{x,h}) + I_x(\varphi_{x,h}) \leq \inf_{\varphi \in \mathcal{E}} \{h(\varphi) + I_x(\varphi)\} + \eta/2.
   \end{equation}
   If we let $s_0 = 2 \sup_{h \in L} \|h\|_{C(\mathcal{E})} + \eta/2$, then Lemma \ref{lem:inf-bound} guarantees that $\varphi_{x,h} \in \Phi_x(s_0)$.

   Because $L$ is a equicontinuous set, there exists $\delta>0$ such that for any $h \in L$, $\rho(\varphi,\psi)<\delta$ implies $|h(\varphi) - h(\psi)|< \eta/2$. In particular, for any $x \in A$, $h \in L$, and $\e>0$,
   \begin{align} \label{eq:exp-upper-eulp}
     &\E \left(- \frac{h(X^\e_x)}{a(\e)} \right) \geq \E \left(\exp \left(-\frac{h(X^\e_x)}{a(\e)} \right) \mathbbm{1}_{\{\rho(X^\e_x,\varphi_{x,h})<\delta\}} \right)\nonumber \\
     &\geq \exp\left( - \frac{(h(\varphi_{x,h}) + \eta/2)}{a(\e)} \right) \Pro(\rho(X^\e_x, \varphi_{x,h})<\delta).
   \end{align}

   Therefore, by \eqref{eq:h-upper-eulp}, \eqref{eq:exp-upper-eulp}, and the FWULDP lower bound \eqref{eq:fwuldp-lower},
   \begin{align*}
     &\liminf_{\e \to 0} \inf_{x \in A} \inf_{h \in L} \left(a(\e) \log \E \exp\left(- \frac{h(X^\e_x)}{a(\e)} \right) + \inf_{\varphi \in \mathcal{E}} \{h(\varphi) + I_x(\varphi)\}\right)\\
     &\geq \liminf_{\e \to 0} \inf_{x \in A} \inf_{h \in L} \big(- h(\varphi_{x,h}) - \eta/2 + a(\e) \log \Pro(\rho(X^\e_x,\varphi_{x,h})<\delta)\\
      &\hspace{3cm} + h(\varphi_{x,h}) + I_x(\varphi_{x,h}) - \eta/2 \big)\\
     & \geq \liminf_{\e\to 0} \inf_{x \in A} \inf_{\varphi \in \Phi_x(s_0)} \left(a(\e) \log \Pro(X^\e_x,\varphi) + I_x(\varphi) \right) - \eta\\
     &\geq -\eta.
   \end{align*}
   The EULP lower bound \eqref{eq:eulp-lower} follows because $\eta$ was arbitrary.
 \end{proof}

 \begin{lemma} \label{lem:fw->eulp-up}
   With no extra assumptions, the FWULDP implies the \\EULP upper bound. For any $A \in \mathscr{A}$ and family $L \subset C(\mathcal{E})$ of uniformly bounded, equicontinuous functions from $\mathcal{E} \to \mathbb{R}$,
   \begin{equation} \label{eq:eulp-upper}
     \limsup_{\e \to 0} \sup_{x \in A} \sup_{h \in L} \left(a(\e) \log \E \exp \left( - \frac{h(X^\e_x)}{a(\e)}  \right) + \inf_{\varphi \in \mathcal{E}} \{h(\varphi) + I_x(\varphi)\} \right) \leq 0.
   \end{equation}
 \end{lemma}

 \begin{proof}
   Assume that $X^\e_x$ satisfies a FWULDP with respect to rate function $I_x$ with speed $a(\e)$ uniformly over $\mathscr{A}$.
   Let $A \in \mathscr{A}$ and $L \subset C(\mathcal{E})$ be a family of uniformly bounded equicontinuous functions from $\mathcal{E} \to \mathbb{R}$. Fix $\eta>0$. By  the equicontinuity of $L$, there exists $\delta>0$ such that whenever $\rho(\varphi,\psi)<\delta$ it follows that $|h(\varphi) - h(\psi)|<\eta/2$. Let $s_0 = \sup_{h \in L} 2\|h\|_{C(\mathcal{E})} + \eta$.

    Let $N \in \mathbb{N}$ be such that $ N\eta/2 > s_0$. For $x \in A$, define the subsets of $\mathcal{E}$,
    \begin{align*}
      &E_0^x = \left\{\varphi \in \mathcal{E}: \dist(\varphi, \Phi_x(\eta/2))<\delta \right\}\\
      &E^x_k = \left\{\varphi \in \mathcal{E}: \dist(\varphi, \Phi_x(k\eta/2))\geq \delta, \ \dist(\varphi, \Phi_x((k+1)\eta/2))<\delta \right\}, \\ &\text{ for }k=1,...,N-1\\
      &E^x_N = \left\{\varphi \in \mathcal{E}: \dist(\varphi, \Phi_x(N\eta/2))\geq \delta \right\}.
    \end{align*}
    Note that $\bigcup_{k=0}^N E^x_k = \mathcal{E}$.
    For any $x \in A$, $h \in L$, and $\e>0$,
    \begin{align*}
      &\E \exp \left( -\frac{h(X^\e_x)}{a(\e)}\right) \leq \sum_{k=0}^N \E \exp\left( - \frac{h(X^\e_x)}{a(\e)} \mathbbm{1}_{\{X^\e_x \in E_k^x\}} \right)\\
      &\leq \sum_{k=0}^N \exp\left(-\frac{\inf_{\varphi \in E_k^x} h(\varphi)}{a(\e)} \right) \Pro(X^\e_x \in E_k^x).
    \end{align*}
    It follows that
    \begin{align*}
      & a(\e) \log \E \exp \left( -\frac{h(X^\e_x)}{a(\e)}\right)\nonumber\\
      &\leq a(\e) \log \left(\sum_{k=0}^N \exp\left(-\frac{\inf_{\varphi \in E_k^x} h(\varphi)}{a(\e)} \right) \Pro(X^\e_x \in E_k^x) \right)\nonumber\\
      & \leq a(\e) \log(N+1) + \max_{k \in \{0,...,N\}} \left\{-\inf_{\varphi \in E_k^x}h(\varphi) + a(\e) \log \Pro(X^\e_x \in E_k^x) \right\}.
    \end{align*}
    By adding and subtracting $\frac{k\eta}{2}$,
    \begin{align}\label{eq:log-estimate-2}
      & a(\e) \log \E \exp \left( -\frac{h(X^\e_x)}{a(\e)}\right)\nonumber\\
      & \leq a(\e) \log(N+1) + \max_{k \in \{0,...,N\}} \left\{  a(\e) \log \Pro(X^\e_x \in E_k^x) + k\eta/2 \right\} \nonumber \\
      &\qquad+ \max_{k \in \{0,...,N\}} \left\{-\inf_{\varphi \in E_k^x} h(\varphi) - k\eta/2 \right\}\nonumber\\
      &\leq a(\e) \log(N+1) + \max_{k \in \{0,...,N\}} \left\{  a(\e) \log \Pro(X^\e_x \in E_k^x) + k\eta/2 \right\}\nonumber\\
       &\qquad- \min_{k \in \{0,...,N\}} \left\{\inf_{\varphi \in E_k^x} h(\varphi) + k\eta/2 \right\}.
    \end{align}
    By the definition of $E_k^x$ for $k\in\{1,..,N\}$, $$\Pro(X^\e_x \in E_k^x) \leq \Pro(\dist(X^\e_x,\Phi_x(k\eta/2))\geq\delta)$$ and it follows by the FWULDP upper bound \eqref{eq:fwuldp-upper} that
    \begin{equation*}
      \limsup_{\e \to 0} \sup_{x \in A} \max_{k \in \{1,...,N\}} \{a(\e) \log \Pro(X^\e_x \in E_k^x) + k\eta/2\} \leq 0.
    \end{equation*}
    The $k=0$ case is trivially true because $ \Pro(X^\e_x \in E_k) \leq 1$ so it follows that
    \begin{equation} \label{eq:Ek-estimate-2}
      \lim_{\e \to 0} \sup_{x \in A} \max_{k \in \{0,...,N\}} \{a(\e) \log \Pro(X^\e_x \in E_k^x) + k\eta/2\} \leq 0.
    \end{equation}

    The last step required to prove \eqref{eq:eulp-upper} is to show that for any $x \in A$,
    \begin{equation} \label{eq:inf-leq-inf-2}
      \inf_{\varphi \in \mathcal{E}} \{h(\varphi) + I_x(\varphi)\} \leq \min_{k \in \{0,..,N\}} \left\{\inf_{\varphi \in E_k^x} h(\varphi) + k\eta/2\right\} + \eta.
    \end{equation}
    Fix $k \in \{0,...,N-1\}$. Let $\varphi \in E_k^x$ be arbitrary. By the definition of $E_k^x$, there exists $\tilde\varphi \in \Phi_x((k+1)\eta/2)$ such that $\rho(\tilde{\varphi}, \varphi)< \delta$.  At the beginning of the proof we chose $\delta$ such that the equicontinuity of $L$ implies that for all $h \in L$, $h(\tilde{\varphi})  \leq h(\varphi) + \eta/2$. Note that $I_x(\tilde\varphi) \leq k\eta/2 + \eta/2$. Therefore,
    \begin{equation*}
      (h(\varphi) + k\eta/2) + \eta \geq h(\tilde \varphi) + I_x(\tilde\varphi) \geq \inf_{\phi \in \mathcal{E}} \{h(\phi) + I_x(\phi)\}.
    \end{equation*}
    Take the infimum over $\varphi \in E_k^x$. For any $k \in \{0,...,N-1\}$ and any $x \in A$,
    \begin{equation*}
      \inf_{\varphi \in E_k^x}\left(h(\varphi) + k\eta/2\right) + \eta \geq \inf_{\varphi \in \mathcal{E}} \{h(\varphi) + I_x(\varphi)\}.
    \end{equation*}
    The above inequality also holds for $k=N$ because of Lemma \ref{lem:inf-bound} and our choice of $N$. Therefore, \eqref{eq:inf-leq-inf-2} holds.

    By  \eqref{eq:log-estimate-2}, \eqref{eq:Ek-estimate-2}, and \eqref{eq:inf-leq-inf-2},
    \begin{align*}
      &\limsup_{\e \to 0} \sup_{x \in A} \sup_{h \in L} \left(a(\e)\log \E \exp \left(- \frac{h(X^\e_x)}{a(\e)} \right) + \inf_{\varphi \in \mathcal{E}} \{h(\varphi) + I_x(\varphi)\} \right)\\
      &\leq \limsup_{\e \to 0} \left( a(\e)\log(N+1) +  \sup_{x \in A} \max_{k \in \{0,...,N\}} \{a(\e) \log \Pro(X^\e_x \in E_k^x) + k\eta/2\}
      \right)\\
      &\qquad + \sup_{x \in A} \sup_{h \in L} \left(\inf_{\varphi\in \mathcal{E}} \{h(\varphi) + I_x(\varphi)\} -  \min_{k \in \{0,...,N\}}\left\{\inf_{\varphi \in E_k^x}h(\varphi) + k\eta/2\right\} \right)\\
      & \leq \eta.
    \end{align*}
    The EULP upper bound \eqref{eq:eulp-upper} follows because $\eta>0$ was arbitrary.
 \end{proof}

\section{Proof of Theorem \ref{thm:eulp-sufficient}} \label{S:EULP}
  Assume that Assumption \ref{assum:mathcal-G} holds. This means that for any $\e>0$ and $x \in \mathcal{E}_0$, $X^\e_x = \mathscr{G}_x (\sqrt{\e}\beta)$.
  By the variational principle of \cite[Theorem 2]{bdm-2008}, for any bounded continuous $h: \mathcal{E} \to \mathbb{R}$, $x \in \mathcal{E}_0$ and $\e>0$,
  \begin{align} \label{eq:variational}
    &\e \log \E \exp \left(- \frac{h(X^\e_x)}{\e} \right) =\nonumber\\
     &-\inf_{u \in \mathcal{P}_2} \E\left\{\frac{1}{2}\int_0^T |u(s)|_U^2 ds + h \left(\mathscr{G}^\e_x \left( \sqrt{\e}\beta(\cdot) + \int_0^\cdot u(s)ds \right) \right) \right\}.
  \end{align}
  Similarly, by the definition of the rate function \eqref{eq:rate-fct},
  \begin{align} \label{eq:variational-rate-fct}
    &\inf_{\varphi \in \mathcal{E}} \{I_x(\varphi) + h(\varphi)\}\nonumber\\
     &= \inf_{u \in L^2([0,T]:U)} \left\{\frac{1}{2}\int_0^T |u(s)|_U^2 ds + h\left(\mathscr{G}^0_x\left( \int_0^\cdot u(s)ds \right) \right) \right\}.
  \end{align}

  Let $L\subset C(\mathcal{E})$ be a set of equicontinuous, equibounded functions from $\mathcal{E} \to \mathbb{R}$ and let $A \in \mathscr{A}$. To simplify the notation of the proof, for any $x \in \mathcal{E}_0$, $u \in \mathcal{P}_2$, and $\e\geq 0$ set
  \begin{equation}
    X^{\e,u}_x:= \mathscr{G}^\e_x\left( \sqrt{\e}\beta + \int_0^\cdot u(s)ds\right).
  \end{equation}

  \textbf{Upper bound}

  Let $x_n \in A$, $h_n \in L$, and $\e_n \downarrow 0$ be arbitrary sequences. Let $\eta>0$. Following the localization arguments of \cite[Theorem 4.4]{bd-2000} and the fact that the $h_n$ are equibounded, we can choose $N>0$ large enough and $u_n \in \mathcal{P}_2^N$  satisfying (see \eqref{eq:variational})
  \begin{align*}
    &\e_n \log \E \exp \left(-\frac{h_n(X^{\e_n}_{x_n})}{\e_n} \right)\\
    &= -\inf_{u \in \mathcal{P}_2}\E \left\{\frac{1}{2}\int_0^T |u(s)|_U^2 ds + h_n \left( 
    X^{\e_n,u}_{x_n}\right) \right\}\\
    &\leq -\E \left\{\frac{1}{2}\int_0^T |u_n(s)|_U^2 ds + h_n \left( 
    X^{\e_n,u_n}_{x_n}\right) \right\} + \eta/3.
  \end{align*}

  Then because the right-hand side of \eqref{eq:variational-rate-fct} is an infimum,
  \begin{align*}
    &\inf_{\varphi \in \mathcal{E}} \left\{I_{x_n}(\varphi) + h_n(\varphi)\right\}\\
    &=\inf_{u \in L^2([0,T]:U)} \left\{\frac{1}{2}\int_0^T |u(s)|_U^2 ds + h_n\left(
    X^{0,u}_{x_n}\right) \right\}\\
    &\leq \E\left\{\frac{1}{2}\int_0^T |u_n(s)|_U^2 ds + h_n\left(
    X^{0,u_n}_{x_n}\right) \right\}.
  \end{align*}

  By these estimates,
  \begin{align} \label{eq:upper-diff}
    &\e_n \log \E \exp \left(- \frac{h_n(X^{\e_n}_{x_n})}{\e_n} \right) + \inf_{\varphi\in\mathcal{E}} \{I_{x_n}(\varphi) + h_n(\varphi) \} \nonumber\\
    &\leq
      -\E \left\{\frac{1}{2}\int_0^T |u_n(s)|_U^2 ds + h_n \left( 
      X^{\e_n,u_n}_{x_n}\right) \right\}  \nonumber\\&\qquad
      + \E\left\{\frac{1}{2}\int_0^T |u_n(s)|_U^2 ds + h_n\left(
      X^{0,u_n}_{x_n}\right) \right\} + \eta/3
     \nonumber\\
    &\leq  \E \left\{-h_n \left(
    X^{\e_n,u_n}_{x_n}\right) +  h_n\left(
    X^{0,u_n}_{x_n}\right) \right\} + \eta/3.
  \end{align}
  Because the family $L$ is assumed to be bounded and equicontinuous, there exists $M>0$ and $\delta>0$ such that for all $n \in \NN$,
  \begin{equation} \label{eq:hn-bded-equicont}
    \|h_n\|_{C(\mathcal{E})} \leq M \text{ and if }\rho(\varphi,\psi)\leq\delta \text{ then } |h_n(\varphi) - h_n(\psi)|\leq \eta/3.
  \end{equation}
  This means that for any two $\mathcal{E}$-valued random variables,
  \[\E|h_n(X_1) - h_n(X_2)| \leq 2M \Pro(\rho(X_1,X_2)>\delta) + \frac{\eta}{3}\Pro(\rho(X_1,X_2)\leq \delta).\]
   In particular, \eqref{eq:upper-diff} guarantees that
  \begin{align}
    &\limsup_{n \to +\infty} \left(\e_n \log \E \exp \left(- \frac{h_n(X^{\e_n}_{x_n})}{\e_n} \right) + \inf_{\varphi\in\mathcal{E}} \{I_{x_n}(\varphi) + h_n(\varphi) \}\right) \nonumber\\
    &\leq \limsup_{n \to +\infty}  2M \Pro\left(\rho  \left(
    X^{\e_n,u_n}_{x_n}, X^{0,u_n}_{x_n}\right)>\delta
    \right) + \frac{2\eta}{3}.
  \end{align}
  Assumption \ref{assum:mathcal-G} guarantees that
  \[\lim_{n \to +\infty}\Pro\left(\rho\left(
  X^{\e_n,u_n}_{x_n}, X^{0,u_n}_{x_n}\right)>\delta\right)=0.\]
  Therefore,
  \[\limsup_{n \to +\infty} \left(\e_n \log \E \exp \left(- \frac{h_n(X^{\e_n}_{x_n})}{\e_n} \right) + \inf_{\varphi\in\mathcal{E}} \{I_{x_n}(\varphi) + h_n(\varphi) \}\right)\leq \eta.\] The EULP upper bound follows because the sequences $(h_n, x_n, \e_n)$ and $\eta>0$ were arbitrary.

  \textbf{Lower bound}

  The proof of the EULP lower bound is almost exactly the same as that of the upper bound.
  Let $x_n \in A$, $h_n \in L$, and $\e_n \downarrow 0$ be arbitrary. Fix $\eta>0$. Lemma \ref{lem:inf-bound} along with the definition of the rate function \eqref{eq:rate-fct} guarantee that we can choose $N>0$ large enough and find and $u_n\in \mathcal{S}^N$ so that by \eqref{eq:variational-rate-fct},
  \begin{align*}
    &\inf_{\varphi \in \mathcal{E}} \{I_{x_n}(\varphi) + h_n(\varphi)\} = \inf_{u \in L^2([0,T]:U)} \left\{\frac{1}{2}\int_0^T |u(s)|_U^2 ds + h_n\left(
    X^{0,u}_{x_n}\right) \right\}\\
    & \geq \frac{1}{2}\int_0^T |u_n(s)|_U^2 ds + h_n\left(
    X^{0,u_n}_{x_n}\right) -\frac{\eta}{3}.
  \end{align*}

  Because the right-hand side of \eqref{eq:variational} includes an infimum,
  \begin{align*}
    &\e_n \log \E \exp \left(- \frac{h_n(X^{\e_n}_{x_n})}{\e_n} \right)
     = -\inf_{u \in \mathcal{P}_2}\E \left\{\frac{1}{2}\int_0^T |u(s)|_U^2 ds + h_n \left(
     X^{\e_n,u}_{x_n}\right) \right\}\\
    &\geq -\E\left\{\frac{1}{2}\int_0^T |u_n(s)|_U^2 ds + h_n \left(
    X^{\e_n,u_n}_{x_n}\right)\right\}.
  \end{align*}
  Combining these estimates and remembering that the chosen $u_n$ are non-random,
  \begin{align*}
    &\e_n \log \E \exp \left(- \frac{h_n(X^{\e_n}_{x_n})}{\e_n} \right) + \inf_{\varphi \in \mathcal{E}} \{I_{x_n}(\varphi) + h_n(\varphi)\}  \\
    &\geq \frac{1}{2}\int_0^T |u_n(s)|_U^2 ds + h_n\left(
    X^{0,u_n}_{x_n}\right) -\frac{\eta}{3} \\
     &\qquad-\E\left\{\frac{1}{2}\int_0^T |u_n(s)|_U^2 ds + h_n \left(
     X^{\e_n,u_n}_{x_n}\right)\right\}\\
    &\geq  h_n\left(
    X^{0,u_n}_{x_n}\right) - \E h_n \left(
    X^{\e_n,u_n}_{x_n}\right) -\frac{\eta}{3}.
  \end{align*}
  Because $L$ is a family of bounded and equicontinuous functions, there exists $M\geq 0$ and $\delta>0$ such that \eqref{eq:hn-bded-equicont} holds. In particular,
  \begin{align*}
    &h_n\left(
    X^{0,u_n}_{x_n}\right) - \E h_n \left(
    X^{\e_n,u_n}_{x_n}\right)
    \geq -2M \Pro\left(\rho \left(
    X^{0,u_n}_{x_n} , 
    X^{\e_n,u_n}_{x_n}\right)>\delta \right) - \frac{\eta}{3}.
  \end{align*}
  By Assumption \ref{assum:mathcal-G},
  \[\lim_{n \to +\infty} \Pro\left(\rho \left(
    X^{0,u_n}_{x_n} , 
    X^{\e_n,u_n}_{x_n}\right)>\delta \right) = 0.\]
  Therefore,
  \[\liminf_{n \to +\infty} \left(\e_n \log \E \exp \left(- \frac{h_n(X^{\e_n}_{x_n})}{\e_n} \right) + \inf_{\varphi \in \mathcal{E}} \{I_{x_n}(\varphi) + h_n(\varphi)\}  \right) \geq -\eta.\]
  Theorem \ref{thm:eulp-sufficient} follows because the sequences $(h_n, x_n, \e_n)$ and $\eta>0$ chosen were arbitrary.

\begin{appendix}
  \section{Some properties of rate function} \label{S:properties}
In this appendix, we collect some useful results on the properties of rate functions. The first result in this section says that if $I_x$ is a large deviations rate function for a collection of $\mathcal{E}$-valued random variables $X^\e_x$ then $\inf_{\varphi \in \mathcal{E}} I_x(\varphi) = 0$.

\begin{lemma} \label{lem:inf-I}
  Fix $x \in \mathcal{E}_0$ and suppose that $\{X^\e_x\}$ is a collection of $\mathcal{E}$-valued random variables and $I_x$ is a rate function. Assume that either
  \begin{enumerate}[(a)]
    \item For any closed set $F \subset \mathcal{E}$,
    \[\limsup_{\e \to 0} a(\e) \log \Pro(X^\e_x \in F) \leq -\inf_{\varphi \in F}I_x(\varphi),\]
    \item For any $\delta>0$ and $s_0>0$,
    \[\limsup_{\e \to 0} \sup_{s \in [0,s_0]} \left(a(\e) \log\Pro(\dist(X^\e_x, \Phi_x(s))\geq \delta) + s\right)\leq 0,\]
    or
    \item For any bounded continuous $h : \mathcal{E} \to \mathbb{R}$,
    \[\limsup_{\e \to 0} \left|a(\e) \log \E \exp \left( - \frac{h(X^\e_x)}{a(\e)} \right) + \inf_{\varphi \in \mathcal{E}} \{h(\varphi) + I_x(\varphi)\} \right|=0.\]
  \end{enumerate}
  Then it follows that
  \begin{equation} \label{eq:inf-I}
    \inf_{\varphi \in \mathcal{E}} I_x(\varphi) = 0.
  \end{equation}
  Furthermore, if $I_x$ is a good rate function, then there exists $\tilde\varphi \in \mathcal{E}$ such that $I_x(\tilde\varphi) = 0$.
\end{lemma}

\begin{proof}
  This topic is discussed in \cite[Chapter 1.2]{dz} under assumption (a). In this proof, we show that the result is true regardless of the definition of large deviations principle that we use. If (a) is true, then noting that $F= \mathcal{E}$ is a closed set $\Pro(X^\e_x \in \mathcal{E}) = 1$ so
  \[-\inf_{\varphi \in \mathcal{E}} I_x(\varphi) \geq \limsup_{\e \to 0} a(\e) \log \Pro(X^\e_x \in \mathcal{E}) = 0.\]

  If (b) holds, then we prove the result by contradiction. Assume by contradiction that $\inf_{\varphi \in \mathcal{E}}  I_x(\varphi) = s>0$. This means that the level set $\Phi_x(s/2) = \emptyset$. Then for any $\varphi \in \mathcal{E}$, $\dist(\varphi, \Phi_x(s/2)) = +\infty$. Therefore for any $\delta>0$, $\Pro(\dist(X^\e_x, \Phi_x(s/2))\geq \delta) = 1$. This contradicts (b) because $$a(\e) \log \Pro(\dist(X^\e_x, \Phi_x(s/2))\geq \delta) + s/2 = s/2 >0.$$

  If (c) holds, then we set $h$ to be the constant function $h(\psi) \equiv 0$. Then $\E \exp \left(-\frac{h(X^\e_x)}{a(\e)}\right) = 1$ and $0=\inf_{\varphi \in \mathcal{E}} \{h(\varphi) + I_x(\varphi)\} = \inf_{\varphi \in \mathcal{E}} I_x(\varphi)$ proving the result.

  Finally, if $I_x$ is a good rate function then the minimum is attained. Specifically, we can find a sequence $\varphi_n \in \mathcal{E}$ such that $\lim_{n \to +\infty}I_x(\varphi_n)=0$. By the compactness of level sets, a subsequence converges to a limit $\tilde{\varphi}$. This limit has the property that for any $\delta>0$, $I_x(\tilde\varphi)<\delta$. Therefore $I_x(\tilde \varphi) = 0$.
\end{proof}

\begin{lemma} \label{lem:inf-bound}
  For any $x \in \mathcal{E}_0$, suppose that $\{X^\e_x\}$ satisfies a large deviations principle with respect to the rate function $I_x$. Let $h: \mathcal{E} \to \mathbb{R}$ be a bounded and continuous function.
  \begin{enumerate}
    \item It follows that
    \begin{equation} \label{eq:h-I-sup-bound}
      \inf_{\varphi \in \mathcal{E}} \{h(\varphi) + I_x(\varphi)\} \leq \|h\|_{C(\mathcal{E})}.
     \end{equation}
    \item If $\varphi^\delta \in \mathcal{E}$ is such that
      \begin{equation*}
        h(\varphi^\delta) + I_x(\varphi^\delta) \leq \inf_{\varphi \in \mathcal{E}} \{h(\varphi) + I_x(\varphi)\}+\delta.
      \end{equation*}
      then $\varphi^\delta \in \Phi_x(2\|h\|_{C(\mathcal{E})} + \delta)$.
    \item If $I_x$ is a good rate function, then there exists $\varphi^0 \in \Phi_x(2\|h\|_{C(\mathcal{E})})$ such that
       \begin{equation} \label{eq:h-I-0}
         h(\varphi^0) + I_x(\varphi^0) = \inf_{\varphi \in \mathcal{E}} \{h(\varphi) + I_x(\varphi)\}.
       \end{equation}
  \end{enumerate}
\end{lemma}
\begin{proof}
  By Lemma \ref{lem:inf-I},
  \[\inf_{\varphi \in \mathcal{E}} \{h(\varphi) + I_x(\varphi)\} \leq \|h\|_{C(\mathcal{E})} + \inf_{\varphi \in \mathcal{E}} I_x(\varphi) \leq \|h\|_{C(\mathcal{E})}\]
  proving \eqref{eq:h-I-sup-bound}.
  For any $\delta>0$, there exists $\varphi^\delta \in \mathcal{E}$ such that
  \[h(\varphi^\delta) + I_x(\varphi^\delta) \leq \inf_{\varphi \in \mathcal{E}} \{h(\varphi) + I_x(\varphi)\} + \delta \leq \|h\|_{C(\mathcal{E})} + \delta.\]
  It follows that,
  \[I_x(\varphi^\delta) \leq 2 \|h\|_{C(\mathcal{E})} + \delta\]
  proving that $\varphi^\delta \in \Phi_x(2\|h\|_{C(\mathcal{E})} + \delta)$.

  If $I_x$ is a good rate function, then by the compactness of the level sets, there exists a subsequence $\delta_n \to 0$ and a limit $\varphi^0$ such that $\varphi^{\delta_n} \to \varphi^0$ and $I_x(\varphi^0) \leq \liminf_{ n \to +\infty} I_x(\varphi^{\delta_n})$. Because $h$ is continuous $h(\varphi^{\delta_n}) \to h(\varphi^0)$ and
  \[h(\varphi^0) + I_x(\varphi^0) = \inf_{\varphi \in \mathcal{E}} \{h(\varphi) + I_x(\varphi)\}\]
  and $\varphi^0 \in \Phi_x(2\|h\|_{C(\mathcal{E})})$.
\end{proof}

\section{Sketch of proofs of Lemmas \ref{lem:Gamma}, \ref{lem:Lambda}, and \ref{lem:Theta}} \label{S:lemmas}
\begin{proof}[Sketch of proof of Lemma \ref{lem:Gamma}]
  By the factorization method of \cite[Chapter 5.3.1]{dpz}, for any $\mathcal{F}_t$-adapted $\varphi \in C([0.T]:H)$,
  \[\Gamma(\varphi)(t)  = \frac{\sin(\pi \alpha)}{\pi} \int_0^t(t-s)^{\alpha-1} S(t-s) \Gamma_\alpha(\varphi)(s)ds\]
  where
  \[\Gamma_\alpha(\varphi)(t) = \int_0^t (t-s)^{-2\alpha}S(t-s)G(\varphi(s))dw(s).\]
  By the Burkholder-Davis-Gundy inequality, for $\varphi \in C([0,T]:H)$, $t>0$,
  \[\E|\Gamma_\alpha(\varphi)(t) |_H^p \leq \E \left(\int_0^t (t-s)^{-2\alpha}\|S(t-s)G(\varphi(s)) \|_{L_2}^2 ds \right)^{\frac{p}{2}}.\]

  If \eqref{eq:G-bounded} holds, then for $p>\frac{1}{\alpha}$, by \eqref{eq:K-int}
  \[\E|\Gamma_\alpha(\varphi)(t) |_H^p \leq \E \left(\int_0^t s^{-2\alpha} K^2(s)ds \right)^{\frac{p}{2}} \leq C .\]
  Then by \cite[equation (5.13)]{dpz}, for $p> \frac{1}{\alpha}$
  \[\E|\Gamma(\varphi) |_{C([0,t]:H)}^p \leq C \E \int_0^t |\Gamma_\alpha(\varphi)(s)|_H^p  ds\leq C .\]

  If \eqref{eq:G-grows} holds, then for $p> \frac{1}{\alpha}$, by \eqref{eq:K-int}
  \begin{align*}
    &\E|\Gamma_\alpha(\varphi)(t) |_H^p \leq C\E \left(1 +|\varphi|_{C([0,t]:H)}\right)^p  \left(\int_0^t s^{-2\alpha} K^2(s)ds \right)^{\frac{p}{2}}\\
     &\leq C \E\left(1 + |\varphi|_{C([0,t]:H)}^p \right).
  \end{align*}
  Then
  \[\E|\Gamma(\varphi) |_{C([0,t]:H)}^p \leq C \E \int_0^t |\Gamma_\alpha(\varphi)(s)|_H^p  ds\leq C \left( 1 +\E \int_0^t |\varphi|_{C([0,s]:H)}^pds\right).\]
  The proof for \eqref{eq:Gamma-Lip} is the same.
\end{proof}

\begin{proof}[Sketch of proof of Lemma \ref{lem:Lambda}]
  Fix $T>0$ and let $\varphi \in C([0,T]:H)$ and $u \in L^2([0,T]:U)$. We once again use the factorization method \cite[Chapter 5.3.1]{dpz} and observe that
  \begin{equation} \label{eq:Lambda-factor}
    [\Lambda(\varphi)u](t) = \frac{\sin(\pi\alpha)}{\pi} \int_0^t (t-s)^{\alpha - 1} S(t-s) \Lambda^\alpha_s(\varphi)uds
  \end{equation}
  where
  \begin{equation} \label{eq:Lambda-alpha}
    \Lambda^\alpha_t(\varphi)u = \int_0^t (t-s)^{-\alpha} S(t-s) G(\varphi(s))u(s)ds.
  \end{equation}
   Let $\{e_k\}_{k=1}^\infty$ be a complete orthonormal basis for $U$.  For any $t>0$,
  \begin{align*}
    &\Lambda^\alpha_t(\varphi)u = \int_0^t(t-s)^{-\alpha} S(t-s)G(\varphi(s))u(s)ds \\
    &= \sum_{k=1}^\infty \int_0^t (t-s)^{-\alpha}S(t-s)G(\varphi(s))e_k \left<u(s),e_k\right>_Uds.
  \end{align*}
  By H\"older's inequality,
  \begin{align*}
    &|\Lambda^\alpha_t(\varphi)u|_H \\
    &\leq \left(\sum_{k=1}^\infty \int_0^t (t-s)^{-2\alpha}|S(t-s)G(\varphi(s))e_k|_H^2 \right)^{\frac{1}{2}} \left(\sum_{k=1}^\infty \int_0^t \left<u(s),e_k\right>_U^2 ds \right)^{\frac{1}{2}} \\
    &\leq |u|_{L^2([0,t]:H)} \left(\int_0^t (t-s)^{-2\alpha}\|S(t-s)G(\varphi(s))\|_{L_2(U,H)}^2 ds \right)^{\frac{1}{2}}
  \end{align*}
  If \eqref{eq:G-bounded} holds, then
  \[|\Lambda^\alpha_t(\varphi)u|_H \leq |u|_{L^2([0,T]:U)} \left(\int_0^t s^{-2\alpha}K^2(s)ds \right)^{\frac{1}{2}} \leq C |u|_{L^2([0,T]:U)}.\]
  By the factorization formula \eqref{eq:Lambda-factor} and \cite[equation (5.13)]{dpz},
  \[|\Lambda(\varphi)u|_{C([0,T]:H)}^p \leq C \int_0^T |\Lambda^\alpha_t(\varphi)u|_H^pds \leq C\]
  proving \eqref{eq:Lambda-bound}. The proofs for \eqref{eq:Lambda-Lip} and \eqref{eq:Lambda-grow} are analogous.
\end{proof}

\begin{proof}[Sketch of proof of Lemma \ref{lem:Theta}]
 These results are a straightforward consequence of Assumption \ref{assum:B-Lip} and the fact that $S(t)$ is a $C_0$ semigroup. Because $S(t)$ is a $C_0$ semigroup, the mapping $t \mapsto \Theta(\varphi)(t)$ is continuous. Let $\varphi, \psi \in C([0,T]:H)$. Then
 \begin{align*}
   |\Theta(\varphi)(t) - \Theta(\psi)(t)|_H \leq \int_0^t \|S(t-s)\|_{\mathscr{L}(H)} |B(\varphi(s)) - B(\psi(s))|_H ds.
 \end{align*}
 Because $S(t)$ is a $C_0$ semigroup, $\sup_{s \in [0,T]} \|S(s)\|_{\mathscr{L}(H)}<+\infty$ \cite[Theorem 1.2.2]{pazy}. Therefore, it follows from Assumption \ref{assum:B-Lip} that
 \[|\Theta(\varphi)(t) - \Theta(\psi)(t)|_H \leq C\int_0^t |\varphi(s) - \psi(s)|_H ds.\]
 Then \eqref{eq:Theta-Lip} follows by the H\"older inequality. The proof for \eqref{eq:Theta-grow} is the same.
\end{proof}

\section{Proof of Theorem \ref{thm:good-rate-function}} \label{S:good}
As observed in Remark \ref{rem:good}, the rate function $I_x$ is good if for any $N>0$ the level set
\begin{equation} \label{eq:control-set}
  \Phi_x(N) = \left\{X^{0,u}_x: u \in \mathcal{S}^{2N} \right\}
\end{equation}
is a compact subset of $\mathcal{E}$. By Alaoglu's Theorem \cite[Chapter 15.1]{fitzpatrick}, $\mathcal{S}^{2N}$ is a compact metric space under the weak topology on $L^2([0,T]:H)$. We will prove compactness of \eqref{eq:control-set} by proving that whenever $u_n \rightharpoonup u$ in the weak topology on $\mathcal{S}^{2N}$ $X^{0,u_n}_x \to X^{0,u}_x$.

\begin{lemma} \label{lem:Lambda-compact}
  Let $\Lambda$ be defined in \eqref{eq:Lambda}. For any $\varphi \in C([0,T]:H)$, the mapping $u \mapsto \Lambda(\varphi)u$ is continuous from the weak topology on $\mathcal{S}^N$ to the norm topology on $C([0,T]:H)$
\end{lemma}
\begin{proof}
  Choose any $T>0$ and $\varphi \in C([0,T]:H)$. First consider the operator $\Lambda^\alpha_t(\varphi)$ defined in \eqref{eq:Lambda-alpha}. We claim that for any $t \in [0,T]$, $\Lambda^\alpha_t(\varphi)$ is a Hilbert-Schmidt operator (and therefore a compact operator) from $L^2([0,T]:U)$ to $H$. To prove the claim, let $\{e_k\}_{k=1}^\infty$ be a complete orthonormal basis of $U$ and let $\{\phi_j\}_{j=1}^\infty$ be a complete orthonormal basis of $L^2([0,T])$. In this way, $\{e_k \phi_j\}_{k,j=1}^\infty$ is a complete orthonormal basis of $L^2([0,T]:U)$. To prove that $\Lambda^{\alpha}_t(\varphi)$ is Hilbert-Schmidt we calculate that
  \begin{align*}
    &\|\Lambda^\alpha_t(\varphi)\|_{L_2(L^2([0,T]:U), H)}^2\\
    &=\sum_{k=1}^\infty \sum_{j=1}^\infty |\Lambda^\alpha_t(\varphi) e_k \phi_j|_H^2\\
    &=\sum_{k=1}^\infty \sum_{j=1}^\infty \left|\int_0^t (t-s)^{-\alpha} S(t-s)G(\varphi(s))e_k \phi_j(s)ds \right|_H^2.
  \end{align*}
  Let $\{f_i\}_{i=1}^\infty$ be a complete orthonormal basis of $H$. Then the above expression equals
  \begin{align*}
    &= \sum_{k=1}^\infty \sum_{j=1}^\infty \sum_{i=1}^\infty \left(\int_0^t (t-s)^{-\alpha} \left<S(t-s)G(\varphi(s))e_k,f_i\right>_H \phi_j(s) ds \right)^2.
  \end{align*}
  Because $\{\phi_j\}_{j=1}^\infty$ is a complete orthonormal basis for $L^2([0,T])$ and $\{f_i\}_{i=1}^\infty$ is a complete orthonormal basis of $H$, this equals
  \begin{align*}
    &= \sum_{k=1}^\infty  \sum_{i=1}^\infty \int_0^t (t-s)^{-2\alpha}\left<S(t-s)G(\varphi(s))e_k, f_i \right>_H^2 ds\\
    &=\sum_{k=1}^\infty \int_0^t (t-s)^{-2\alpha} |S(t-s)G(\varphi(s))e_k|_H^2\\
    &= \int_0^t (t-s)^{-2\alpha} \|S(t-s)G(\varphi(s))\|_{L_2(U,H)}^2ds.
  \end{align*}
  This is finite by Assumption \ref{assum:G-Lip} proving that $\Lambda^\alpha_t(\varphi)$ is a Hilbert-Schmidt operator from $L^2([0,T]:U)$ to $H$.

  Hilbert-Schmidt operators are compact operators. Compact operators are continuous from the weak topology to the norm topology \cite[Proposition VI.3.3]{conway}. This means that for any sequence $u_n \rightharpoonup u$ in the weak topology on $L^2([0,T]:U)$, and any $t \in [0,T]$,
  \begin{equation}\label{eq:weak-limit}
    \lim_{n \to +\infty} \left|\Lambda^\alpha_t(\varphi)(u_n-u)\right|_H = 0.
  \end{equation}

  By the factorization method \eqref{eq:Lambda-factor} and \cite[equation (5.13)]{dpz}, for $p>\frac{1}{\alpha}$,
  \[|\Lambda(\varphi)(u_n-u)|_{C([0,T]:H)}^p \leq C \int_0^T |\Lambda^\alpha_t(\varphi)(u_n-u)|^pdt.\]
  This converges to zero by \eqref{eq:weak-limit} and the dominated convergence theorem (domination due to Lemma \ref{lem:Lambda}).
\end{proof}

\begin{proof}[Proof of Theorem \ref{thm:good-rate-function}]
  Let $u_n$ be an arbitrary sequence in $\mathcal{S}^{2N}$. Because $S^{2N}$ is weakly compact in $L^2([0,T]:U)$, there exists a subsequence (relabeled $u_n)$ and a limit $u$ such that $u_n \rightharpoonup u$ in the weak topology. We show that $X^{0,u_n}_x \to X^{0,u}_x$ in $C([0,T]:H)$. This proves compactness because the original sequence $u_n$ was arbitrary and every element of $\Phi_x(N)$ can be written as $X^{0,u}_x$ for some $u \in S^{2N}$. Observe that
  \[X^{0,u_n}_x - X^{0,u}_x = \Theta(X^{0,u_n}_x) - \Theta(X^{0,u}_x) + \Lambda(X^{0,u_n}_x)u_n - \Lambda(X^{0,u}_x)u.\]
  We rewrite this as
  \begin{align*}
    X^{0,u_n}_x - X^{0,u}_x =   \Theta&(X^{0,u_n}_x) - \Theta(X^{0,u}_x) + \Lambda(X^{0,u_n}_x)u_n - \Lambda(X^{0,u}_x)u_n \\
    &+ \Lambda(X^{0,u}_x)(u_n-u).
  \end{align*}
  By  \eqref{eq:Lambda-Lip} and \eqref{eq:Theta-Lip}, for $t \in [0,T]$, taking into account that $u_n, u \in \mathcal{S}^{2N}$,
  \begin{align*}
    |X^{0,u}_x - X^{0,u_n}_x|^p_{C([0,t]:H)} \leq C&| \Lambda(X^{0,u}_x)(u_n-u)|_{C([0,t]:H)}^p \\
    &+ C(1+(2N)^{p/2})\int_0^t|X^{0,u}_x - X^{0,u_n}_x|^p_{C([0,s]:H)}ds.
  \end{align*}
  By Gr\"onwall's inequality,
  \[|X^{0,u}_x - X^{0,u_n}_x|^p_{C([0,T]:H)}\leq C e^{C(1 + (2N)^{p/2})T} | \Lambda(X^{0,u}_x)(u_n-u)|_{C([0,t]:H)}^p. \]
  This converges to zero by Lemma \ref{lem:Lambda-compact} proving that $\Phi_x(N)$ is compact and that $I_x$ is a good rate function.
\end{proof}

\end{appendix}

\section*{Acknowledgements}
The author thanks Amarjit Budhiraja, Paul Dupuis, and David Lipshutz for many helpful discussions. The author especially thanks Budhiraja and Dupuis for sharing insights on Lemmas \ref{lem:ulp->fw-low} and \ref{lem:ulp->fw-up}.

\bibliographystyle{imsart-number}
\bibliography{2017}

\end{document}